\documentclass{article}
\usepackage[T1]{fontenc}
\usepackage[utf8]{inputenc}
\usepackage[margin=1.125in]{geometry}
\usepackage{lineno}
\usepackage{amssymb,amsmath}
\usepackage{lipsum}
\usepackage{color}
\usepackage{subfigure}
\usepackage{epsfig}
\usepackage{epstopdf}
\usepackage{graphicx}
\usepackage{algpseudocode}
\usepackage{algorithm}
\usepackage{hyperref}
\usepackage{amsopn}
\usepackage{mathrsfs}
\usepackage[normalem]{ulem}
\ifpdf
  \DeclareGraphicsExtensions{.eps,.pdf,.png,.jpg}
\else
  \DeclareGraphicsExtensions{.eps}
\fi

\usepackage{enumitem}
\setlist[enumerate]{leftmargin=.5in}
\setlist[itemize]{leftmargin=.5in}

\usepackage{amsthm}

\numberwithin{equation}{section}
\numberwithin{algorithm}{section}
   \newtheorem{theorem}{Theorem}[section]
   \newtheorem{lemma}[theorem]{Lemma}
   \newtheorem{corollary}[theorem]{Corollary}
   \newtheorem{proposition}[theorem]{Proposition}
   \newtheorem{definition}[theorem]{Definition}
\newtheorem{remark}[theorem]{Remark}
\newtheorem{example}[theorem]{Example}

\newcommand{\norm}[1]{\lVert#1\rVert}

\newcommand{\fft}{{\rm{fft}}}

\newcommand{\tnn}{{\rm{tnn}}}

\newcommand{\st}{\mbox{s.t.}}

\DeclareMathOperator*{\circc}{bcirc}

\DeclareMathOperator*{\unfold}{unfold}

\newcommand{\fold}{{\rm{fold}}}
\newcommand{\tr}{{\rm{Tr}}}
\newcommand{\sss}{{\mathscr{S}}}

\DeclareMathOperator*{\sgn}{sgn}
\DeclareMathOperator*{\diag}{diag}
\DeclareMathOperator{\prox}{prox}
\DeclareMathOperator*{\argmin}{argmin}
\DeclareMathOperator*{\sqz}{squeeze}

\newcommand{\va}{\mathbf{a}}
\newcommand{\vb}{\mathbf{b}}
\newcommand{\vc}{\mathbf{c}}

\newcommand{\ve}{\mathbf{e}}

\newcommand{\vu}{\mathbf{u}}
\newcommand{\vv}{\mathbf{v}}
\newcommand{\vw}{\mathbf{w}}
\newcommand{\vx}{\mathbf{x}}
\newcommand{\vy}{\mathbf{y}}
\newcommand{\vz}{\mathbf{z}}

\newcommand{\cA}{\mathcal{A}}
\newcommand{\cB}{\mathcal{B}}
\newcommand{\cC}{\mathcal{C}}

\newcommand{\cE}{\mathcal{E}}

\newcommand{\cG}{\mathcal{G}}
\newcommand{\cH}{\mathcal{H}}

\newcommand{\cS}{\mathcal{S}}

\newcommand{\cU}{\mathcal{U}}
\newcommand{\cV}{\mathcal{V}}
\newcommand{\cW}{\mathcal{W}}
\newcommand{\cX}{\mathcal{X}}
\newcommand{\cY}{\mathcal{Y}}
\newcommand{\cZ}{\mathcal{Z}}

\newcommand{\R}{\mathbb{R}}
\newcommand{\E}{\mathbb{E}}
\newcommand{\N}{\mathbb{N}}

\DeclareMathOperator*{\vect}{vec}

\providecommand{\keywords}[1]{\textbf{Keywords: } #1}

\title{Regularized Kaczmarz Algorithms for Tensor Recovery}

\author{Xuemei Chen\thanks{Department of Mathematics and Statistics, University of North Carolina, Wilmington, NC
  (\texttt{chenxuemei@uncw.edu}).}
\and Jing Qin\thanks{Department of Mathematics, University of Kentucky, Lexington, KY
  (\texttt{jing.qin@uky.edu}).}
  }
\date{}

\begin{document}
\maketitle

\begin{abstract}
Tensor recovery has recently arisen in a lot of application fields, such as transportation, medical imaging and remote sensing. Under the assumption that signals possess sparse and/or low-rank structures, many tensor recovery methods have been developed to apply various regularization techniques together with the operator-splitting type of algorithms. Due to the unprecedented growth of data, it becomes increasingly desirable to use streamlined algorithms to achieve real-time computation, such as stochastic optimization algorithms that have recently emerged as an efficient family of methods in machine learning. In this work, we propose a novel algorithmic framework based on the Kaczmarz algorithm for tensor recovery. We provide thorough convergence analysis and its applications from the vector case to the tensor one. Numerical results on a variety of tensor recovery applications, including sparse signal recovery, low-rank tensor recovery, image inpainting and deconvolution, illustrate the enormous potential of the proposed methods.
\end{abstract}



\section{Introduction}
Tensor is an important tool to represent, analyze and process high-dimensional data. By generalizing vectors and matrices, tensors can be used to represent a variety of data sets in a versatile and efficient way. Besides the traditional low-dimensional signal processing problems such as image restoration, tensor modeling is able to improve complex data analysis and processing by exploiting the hidden relationships among data components.
Recently, tensor recovery has arisen in many application areas, such as transportation systems \cite{tan2013tensor}, medical imaging \cite{zhou2013tensor} and hyperspectral image restoration \cite{fan2017hyperspectral}. The goal is to reconstruct a tensor-valued signal from its measurements with multiple channels which may be degraded by noise, blur and so on. For example, many image restoration problems, e.g., image deblurring, can be cast as a tensor recovery problem by treating an image as a special type of tensor \cite{kilmer2013third}. Moreover, tensor modeling typically imposes the assumption of sparsity and low-rank structures on the underlying tensor signal to be reconstructed, which brings the presence of sparsity-promoted regularizations in the objective function. In this work, we focus on third-order tensor recovery models given consistent linear measurements.

Tensor recovery problems usually involve massive data sets which result in large-scale systems of linear equations. As one of the most important iterative algorithm for solving linear systems, the Kaczmarz algorithm was first proposed by Stephen Kaczmarz \cite{karczmarz1937angenaherte} and then rediscovered as the Algebraic Reconstruction Technique (ART) in computed tomography (CT) \cite{gordon1970algebraic}. Due to its simplicity and efficiency, it has been used and developed in many applications, including ultrasound imaging \cite{andersen1984simultaneous}, seismic imaging \cite{peterson1985applications}, positron emission tomography \cite{herman1993algebraic}, electrical impedance tomography \cite{li2013adaptive} and recently phase retrieval \cite{jeong2017convergence,tan2018phase}.

There is a large amount of work on the interpretations, developments and extensions of the Kaczmarz algorithm. For example, by treating each linear equation as a hyperplane of a high-dimensional space, it can be derived by applying the method of successive projections onto convex sets \cite{censor1997parallel}. With random selection of projections, randomized Kaczmarz algorithm for solving consistent over-determined linear systems with a unique solution is proposed in \cite{strohmer2009randomized}, which can be considered as a special case of stochastic gradient matching pursuit (StoGradMP) \cite{needell2016stochastic}. Convergence analysis of randomized Kaczmarz for noisy and random linear systems can be found in \cite{needell2010randomized,chen2012almost}. In addition, application of coordinate descent to the dual formulation of a linear equality constrained least-norm problem leads to the Kaczmarz algorithm \cite{wright2015coordinate}. Some other Kaczmarz types of methods include accelerated randomized Kaczmarz \cite{eldar2011acceleration}, asynchronous parallel randomized Kaczmarz \cite{liu2014asynchronous}, block Kaczmarz algorithms \cite{popa2004kaczmarz,needell2015randomized,durgin2019sparse}, Kaczmarz method for fusion frame recovery~\cite{CP16}, and greedy randomized Kaczmarz \cite{bai2018greedy}. Similar to the standard version, randomized sparse block Kaczmarz can also be obtained from randomized dual block-coordinate descent \cite{petra2015randomized}.

Motivated by the efficiency of the regularized Kaczmarz algorithm in solving linear systems, we intend to extend it from vectors to tensors.  In this work, we integrate the Kaczmarz algorithm into sparse/low-rank tensor recovery to significantly reduce the computational cost while preserving high accuracy. To simplify the discussion, we focus on third-order tensors which can be further extended to other higher-order tensors. Specifically, we assume that the acquired tensor measurements $\cB\in\mathbb{R}^{N_1\times K\times N_3}$ are related with the sensing tensor $\cA\in\mathbb{R}^{N_1\times N_2\times N_3}$ and the underlying signal in a tensor form $\cX\in\mathbb{R}^{N_2\times K\times N_3}$ via the tensor equation $g(\cX)=\cB$. To recover $\cX$, one can consider the following constrained minimization problem
\begin{equation}
\hat\cX=\argmin_{\cX\in\mathbb{R}^{N_2\times K\times N_3}} f(\cX), \quad\st\quad g(\cX)=\cB.
\end{equation}
Here the objective function $f$ is convex on the real tensor space $\mathbb{R}^{N_2\times K\times N_3}$, and $g$ is a map from $\mathbb{R}^{N_2\times K\times N_3}$ to $\mathbb{R}^{N_1\times K\times N_3}$. In this paper, we assume $g(\cX)=\cA*\cX$ where the symbol ``$*$'' denotes the t-product~\cite{KC11} (see Section \ref{sec:tensor}). To solve this linear constrained minimization problem, we propose a general regularized Kaczmarz tensor algorithm, which involves random projection onto the hyperplane and subgradient descent. The very recent work by Ma and Molitor~\cite{ma2020randomized} also uses the Kaczmarz algorithm, but it focuses on  solving the system $\cA*\cX=\cB$ without minimizing an objective function. We discuss convergence guarantees of the proposed algorithm with a deterministic or random control sequence. The proposed framework is also extended to solve the tensor nuclear norm minimization problem. 
We also discuss the important special cases when the third dimension of tensors is fixed as one
for vector and matrix recovery problems.  Furthermore, numerical experiments in various applications, including sparse vector recovery, image inpainting, low-rank tensor recovery, and single/multiple image deblurring, demonstrate the great potential of the proposed algorithms in terms of computational efficiency. There are three major contributions for this work detailed as follows.
\begin{enumerate}
\item We propose a novel regularized Kaczmarz algorithmic framework for tensor recovery problems with thorough convergence analysis. Due to the difference in data structures, extension of a Kaczmarz type of algorithm to tensors is not trivial. A linear convergence rate is proven for the randomized version. Moreover, we also consider the noisy scenario with a slightly stronger assumption on the objective function. 
\item We provide three important cases of the proposed framework with convergence discussions. In particular, we propose a new algorithm for solving the tensor nuclear minimization problem with convergence guarantees, which is particularly useful in a lot of high-dimensional signal processing problems.
\item Numerical experiments on various signal/image recovery problems have justified the proposed performance, which can be further extended to solve other related application problems in various areas. In addition, the proposed algorithms are friendly to parameter tuning, where the step size can be fixed as one. Batched versions have empirically shown the capability of further improvements.
\end{enumerate}

The rest of the paper is organized as follows. In Section~\ref{sec:pre}, we introduce basic concepts and results in convex optimization, and tensors. Our main results are in Section~\ref{sec:tenrec}, where we propose a tensor recovery method based on Kaczmarz algorithm with the convergence guarantees. For the randomized version of this algorithm, we show linear convergence in expectation, even when in presence of noise. As special cases of tensor recovery, Section~\ref{sec:app} discusses how the proposed framework are applied to solve the vector and matrix recovery problem, and the tensor nuclear norm regularized tensor recovery model. Section \ref{sec:exp} lists various numerical experiments illustrating the efficiency of the proposed algorithms.

\section{Preliminaries}\label{sec:pre}
In this section, we provide clarification of notation and a brief review of fundamental concepts in convex optimization and tensor algebra.

Throughout the paper, we use boldface lowercase letters such as $\vx$ for vectors, capital letters such as $X$ for matrices, and calligraphic letters such as $\cX$ for tensors. The sets of all natural numbers and real numbers are denoted by $\N$ and $\R$, respectively. Given a real number $p\geq1$, the $\ell_p$-norm of a vector $\vx\in\R^N$ is defined as $\|\vx\|_p:=\left(\sum_{i=1}^N |x_i|^p\right)^{1/p}$. Analogous to the vector $\ell_2$-norm, the Frobenius norm of a matrix $X\in\R^{N_1\times N_2}$ is defined as $\|X\|_F=\sum_{i=1}^{N_1}\sum_{j=1}^{N_2}x_{ij}^2$. The trace of a square matrix $X$, denoted by $\tr(X)$, is the sum of all diagonal entries of $X$. Furthermore, the nuclear norm of a matrix $X$, denoted by $\|X\|_*$, is defined as the sum of all the singular values of $X$. For a complex-valued matrix $X$, $X^T$ is its transpose by interchanging the row and column index for each entry, and $X^*$ is its complex conjugate transpose, i.e., performing both transpose and componentwise complex conjugate. For any positive integer $k$, the set $\{1,2,\ldots,k\}$ is denoted by $[k]$. Given a finite set $I$, the cardinality of $I$ is denoted by $|I|$.

For a convex set $V$, $P_V$ denotes the orthogonal projection onto $V$. Given a matrix $A$, $R(A)$ is the row space of $A$, and $\sigma_{\min}(A)$, $\sigma_{\max}(A)$ are the smallest and largest nonzero singular values of $A$, respectively. One can show that
\begin{equation}\label{equ:A}
\sigma_{\min}(A)\|P_{R(A)}(\vx)\|_2\leq\|A\vx\|_2\leq\sigma_{\max}(A)\|P_{R(A)}(\vx)\|_2.
\end{equation}

The \emph{soft thresholding} operator (also known as \emph{shrinkage}) $S_\lambda(\cdot)$ is defined componentwise as
\begin{equation}\label{eqn:shrinkage}
(S_{\lambda}(\vx))_i=\sgn(x_i)\max\{|x_i|-\lambda,0\},
\end{equation}
where $\vx\in\mathbb{R}^N$ and $\sgn(\cdot)$ is the signum function which returns the sign of a nonzero number and zero otherwise. This operator can also be extended for matrices and tensors.

\subsection{Convex Optimization Basics}\label{sec:conv1}
To make the paper self-contained, we present basic definitions and properties about convex functions defined on real vector spaces. By reshaping matrices or tensors as vectors, these concepts and results can be naturally extended to functions defined on real matrix or tensor spaces.

For a continuous function $f: \R^N\rightarrow\R$, its \emph{subdifferential} at $\vx\in\R^N$ is defined as
\[
\partial f(\vx)=\{\vx^*: f(\vy)\geq f(\vx)+\langle \vx^*, \vy-\vx\rangle\mbox{ for any }\vy\in\R^N\}.
\]
It can be shown that $\partial f(\vx)$ is closed and convex in $\R^N$. If, in addition, $f$ is convex, then $\partial f(\vx)$ is nonempty for any $\vx\in\R^N$. In addition, a convex function is called \emph{proper} if its epigraph $\{(\vx,\mu)\,:\,\vx\in\R^N,\,\mu\geq f(\vx)\}$ is nonempty in $\R^{N+1}$.

Furthermore, $f$ is \emph{$\alpha$-strongly convex} for some $\alpha>0$ if for any $\vx,\vy\in\R^N$ and $\vx^*\in\partial f(\vx)$ we have
\begin{equation}\label{equ:alpha}
f(\vy)\geq f(\vx)+\langle \vx^*, \vy-\vx\rangle+\frac{\alpha}{2}\|\vy-\vx\|_2^2.
\end{equation}

\begin{example}
Let $f_1(\vx)=\frac{1}{2}\|\vx\|_2^2$, which is differentiable. In this case, the subdifferential becomes a singleton only consisting of the gradient, i.e., $\partial f_1(\vx)=\{\nabla f_1(\vx)\}=\{\vx\}$, and one can show that $f_1$ is 1-strongly convex.
\end{example}

Moreover, it is easy to show that $h(\vx)+\frac{1}{2}\|\vx\|_2^2$ is 1-strongly convex if $h$ is convex. In what follows, we provide two such examples. For more details and examples, please refer to \cite{bauschke2011convex}.

\begin{example}
Given a positive number $\lambda$, $f_\lambda(\vx)=\lambda\|\vx\|_1+\frac{1}{2}\|\vx\|_2^2$ is 1-strongly convex.
\end{example}

\begin{example}
Let $U$ be a real matrix, $\lambda>0$, then  $f_U(\vx)=\lambda\|U\vx\|_1+\frac{1}{2}\|\vx\|_2^2$ is 1-strongly convex.
\end{example}

The \emph{convex conjugate function} of $f$ at $\vz\in\R^N$ is defined as
\[
f^*(\vz):=\sup_{\vx\in\R^N}\{\langle \vz,\vx\rangle-f(\vx)\}.
\]
It can be shown that $\vz\in\partial f(\vx)$ if and only if $\vx=\nabla f^*(\vz)$ \cite{rockafellar1970convex}.

Following the ideas in \cite{lorenz2014linearized}, we can verify that if $f$ is $\alpha$-strongly convex then the conjugate function $f^*$ is differentiable and for any $\vz,\vw\in\R^N$, the following inequalities hold
\begin{align}
\label{equ:pf*}\|\nabla f^*(\vz)- \nabla f^*(\vw)\|_2\leq\frac{1}{\alpha}\|\vz-\vw\|_2,\\
\label{equ:f*}f^*(\vy)\leq f^*(\vx)+\langle \nabla f^*(\vx), \vy-\vx\rangle +\frac{1}{2\alpha}\|\vy-\vx\|_2^2.
\end{align}

For a convex function $f:\R^N\to\R$, the \emph{Bregman distance} between $\vx$ and $\vy$ with respect to $f$ and $\vx^*\in\partial f(\vx)$ is defined as
\begin{equation}\label{eqn:BregDist}
D_{f,\vx^*}(\vx,\vy):=f(\vy)-f(\vx)-\langle \vx^*, \vy-\vx\rangle.
\end{equation}
Since $\langle \vx, \vx^*\rangle=f(\vx)+f^*(\vx^*)$ if $x^*\in\partial f(\vx)$ \cite{rockafellar1970convex}, the Bregman distance can also be written as
\begin{equation}
\label{equ:Df2}
D_{f,\vx^*}(\vx,\vy)=f(\vy)+f^*(\vx^*)-\langle \vx^*, \vy\rangle.
\end{equation}

Note that for $f_1(\vx)=\frac{1}{2}\|\vx\|_2^2$, we have $D_{f_1, \vx^*}(\vx,\vy)=\frac{1}{2}\|\vx-\vy\|_2^2$. In general, the Bregman distance satisfies the property
\begin{equation}\label{equ:Df}
 \frac{\alpha_f}{2}\|\vx-\vy\|_2^2\leq D_{f,\vx^*}(\vx,\vy)\leq \langle \vx^*-\vy^*, \vx-\vy\rangle.
\end{equation}

The \emph{proximal operator} of $f$ is defined as
\begin{equation}\label{eqn:prox}
\prox_f(\vv)=\argmin_\vx \{f(\vx)+\frac{1}{2}\|\vx-\vv\|_2^2\}.
\end{equation}
Due to the Fenchel's duality, we have
\[
\prox_{h}(\vv)=\nabla f^*(\vv),\quad\mbox{where}\quad f(\vx)=h(\vx)+\frac{1}{2}\norm{\vx}_2^2.
\]
It can be shown that the soft thresholding operator is in fact the proximal operator of the $\ell_1$-norm, i.e., $S_\lambda(\vx)=\prox_{\lambda\norm{\cdot}_1}(\vx)$.

We provide an important lemma related to the Bregman distance, which will be used in proving Theorem \ref{thm:tensornoise}.
\begin{lemma}\label{lem:dd}
If $f:\R^N\to\R$ is convex, continuous and proper, then
\[
D_{f,\vx^*}(\vx,\vy)-D_{f,\vw^*}(\vw,\vy)\leq\langle \vx^*-\vw^*, \vx-\vy\rangle.
\]
for any $\vx^*\in\partial f(\vx)$ and $\vw^*\in\partial f(\vw)$.
\end{lemma}

\begin{proof}
\begin{align*}
&D_{f,\vx^*}(\vx,\vy)-D_{f,\vw^*}(\vw,\vy)\\
=&f(\vy)-f(\vx)-\langle \vx^*,\vy-\vx\rangle-f(\vy)+f(\vw)+\langle \vw^*,\vy-\vw\rangle\\
=&f(\vw)-f(\vx)-\langle \vx^*,\vy-\vx\rangle+\langle \vw^*,\vy-\vw\rangle\\
\leq&\langle \vw^*, \vw-\vx\rangle-\langle \vx^*,\vy-\vx\rangle+\langle \vw^*,\vy-\vw\rangle\\
=&\langle \vw^*, \vy-\vx\rangle-\langle \vx^*,\vy-\vx\rangle=\langle \vx^*-\vw^*, \vx-\vy\rangle
\end{align*}
\end{proof}

Next we will introduce the concept of restricted strong convexity. If $f:\R^N\to\R$ is convex, differentiable and proper, then \eqref{equ:alpha} in the definition of an $\alpha$-strongly convex function becomes
\begin{equation}\label{equ:alpha2}
\langle \nabla f(\vy)-\nabla f(\vx), \vy-\vx\rangle\geq\alpha\|\vy-\vx\|^2_2
\end{equation}
for any $\vx,\vy\in \R^N$. We define a weaker version as follows.

\begin{definition}\label{def:rsc}
Let $f:\R^N\to\R$ be convex differentiable with a nonempty minimizer set $X_f:=\{\vx:f(\vx)\leq f(\vy)\text{ for any }\vy\in\R^N\}$. The function $f$ is \textbf{restricted strongly convex} on a convex set $C\subseteq\R^N$ with $\alpha>0$ if
\begin{equation}
\langle\nabla f(\vy)-\nabla f(\vx), \vy-\vx\rangle\geq\alpha\|\vy-\vx\|^2_2
\end{equation}
for any $\vx\in C$ and $\vy=P_{X_f\cap C}(\vx)$.
\end{definition}

This definition can be found in \cite{schopfer2016linear}. The concept of restricted strong convexity first appeared in \cite{LY13} where $C$ is $\R^N$. We include below a useful lemma about the restricted strong convexity, which will be used for proving Lemma~\ref{lem:tensoressential}.

\begin{lemma}[{{\cite[Lemma 2.2]{schopfer2016linear}}}]\label{lem:rsc}
If $f$ is restricted strongly convex on $C$ with the constant $\alpha$, then
\begin{equation}\label{eqn:lem_rsc}
f(\vx)-\min_{\vx}f(\vx)\leq\frac{1}{\alpha}\|\nabla f(\vx)\|_2^2,
\end{equation}
for any $\vx\in C$.
\end{lemma}

\subsection{Tensor Basics}\label{sec:tensor}
We follow the notation for tensor operators in \cite{kilmer2013third}.
If $\cA\in\R^{N_1\times N_2\times N_3}$ with the $k$th frontal slice $A_k=\cA(:,:,k)$, then we define the \emph{block circulant operator} as follows
\begin{equation}\label{eqn:bcirc}
\circc(\cA):=\begin{bmatrix}
A_1&A_{N_3}&\cdots&A_2\\
A_2&A_1&\cdots &A_3\\
\vdots&\vdots&\ddots&\vdots\\
A_{N_3}&A_{N_3-1}&\cdots &A_1\end{bmatrix}\in\R^{N_1N_3\times N_2N_3}.
\end{equation}
Moreover, we define the operator $\unfold(\cdot)$ and its inversion $\fold(\cdot)$ for the conversion between tensors and matrices
\[
\unfold(\cA)=\begin{bmatrix}
A_1\\
A_2\\
\vdots\\
A_{N_3}\end{bmatrix}\in\R^{N_1N_3\times N_2},\quad
\fold\left(\begin{bmatrix}
A_1\\
A_2\\
\vdots\\
A_{N_3}\end{bmatrix}\right)=\cA.
\]
For degenerate cases, we define the operator $\sqz(\cdot)$ that removes all the dimensions of length one from a tensor. For example, if $\cA\in\R^{3\times1\times 1}$, then $\sqz(\cA)$ returns a three-dimensional vector.
The transpose of $\cA$, denoted by $\cA^T$, is the $N_2\times N_1 \times N_3$ tensor obtained by transposing each of the frontal slices and then reversing the order of transposed frontal slices 2 through $N_3$. We have
$$
\circc(\cA^T)=\left(\circc(\cA)\right)^T.
$$
For the two tensors $\cA, \tilde\cA$ of the same size, we define the inner product and the Frobenius norm for tensors as
$$
\langle\cA, \tilde\cA\rangle:=\sum_{i,j,k}\cA(i,j,k)\tilde\cA(i,j,k),\qquad\|\cA\|_F^2=\langle\cA, \cA\rangle.
$$
One can see that $\langle \cA,\tilde{\cA}\rangle=\langle \unfold(\cA),\unfold(\tilde{\cA})\rangle$ where the right hand side is an inner product of two matrices. If $\cA\in\R^{N_1\times N_2\times N_3},  \cC\in\R^{N_2\times K\times N_3}$, then their \emph{t-product} is the $N_1\times K\times N_3$ tensor given by
$$
\cA*\cC:=\fold\left(\circc(\cA)\unfold(\cC)\right).
$$
We list three important properties about the t-product as follows:
\begin{enumerate}
\item Separability in the first dimension
\begin{equation}\label{eqn:tprod_p1}
(\cA*\cC)(i,:,:) = \cA(i,:,:)*\cC.
\end{equation}
\item Sum separability in the second dimension
\begin{equation}\label{eqn:tprod_p2}
\cA* \cC=\sum_{j=1}^{N_2}\cA(:,j,:)*\cC(j,:,:).
\end{equation}
\item Circular convolution in the third dimension: if $\cA,\cC\in\mathbb{R}^{1\times 1\times N_3}$, then
\begin{equation}\label{eqn:tprod_p3}
\sqz(\cA*\cC)=\mathrm{circ}(\va)\vc 
\end{equation}
where both $\va=\sqz(\cA)$ and $\vc=\sqz(\cC)$ are $N_3$-dimensional vectors. Here $\mathrm{circ}(\va)$ is the circular matrix generated by the vector $\va$, i.e., the reduced case of \eqref{eqn:bcirc} when $\cA\in\mathbb{R}^{1\times 1\times N_3}$.
\end{enumerate}
In Table~\ref{tab:tprod}, we summarize the t-products of tensors with reduced dimensions and their corresponding operators for vectors or matrices.
In particular, based on \eqref{eqn:tprod_p1} and \eqref{eqn:tprod_p3}, the t-product of $\cA\in\mathbb{R}^{1\times 1\times N_3}$ and $\cC\in\mathbb{R}^{1\times K\times N_3}$ can be implemented as
\begin{equation}\label{eqn:tprod_r1}
\sqz(\cA*\cC)=\sqz(\cC)\mathrm{circ}(\sqz(\cA))^T,
\end{equation}
where $\sqz(\cA)\in\mathbb{R}^{ N_3}$ and $\sqz(\cC)\in\mathbb{R}^{K\times N_3}$.
If $\cA\in\mathbb{R}^{N_1\times 1\times N_3}$ and $\cC\in\mathbb{R}^{1\times 1\times N_3}$, then
\begin{equation}\label{eqn:tprod_r2}
\sqz(\cA*\cC)=\sqz(\cA)\mathrm{circ}(\sqz(\cC))^T,
\end{equation}
where $\sqz(\cA)\in\mathbb{R}^{N_1\times N_3}$ and $\sqz(\cC)\in\mathbb{R}^{ N_3}$. Furthermore, if $\cA\in\mathbb{R}^{1\times N_2\times N_3}$ and $\cC\in\mathbb{R}^{N_2\times 1\times N_3}$, then the t-product becomes a sum of vector circular convolutions
\begin{equation}\label{eqn:tprod_r3}
\sqz(\cA*\cC)=\sum_{j=1}^{N_2}\mathrm{circ}(\sqz(\cA)(j,:))\sqz(\cC)(j,:)^T,
\end{equation}
where $\sqz(\cA)(j,:)\in\mathbb{R}^{N_3}$ is the $j$th row of the matrix $\sqz(\cA)$. These properties are particularly useful for representing the image blurring as a t-product; see Section~\ref{subsec:decov} for more details.

\begin{table}\label{tab:tprod}
\centering
\begin{tabular}{c|c|c|c}
\hline\hline
condition & $\cA$ & $\cC$& vector/matrix operator\\ \hline
$N_2=K=N_3=1$ & $\R^{N_1\times 1\times 1}$ & $\R^{1\times 1\times 1}$  & scalar multiplication of a vector\\
$N_1=N_2=N_3=1$ & $\R^{1\times 1\times 1}$ & $\R^{1\times K\times 1}$ & scalar multiplication of a vector\\ \hline
$N_1=K=N_3=1$ & $\mathbb{R}^{1\times N_2\times 1}$ & $\mathbb{R}^{N_2\times 1\times 1}$  & inner product of two vectors\\
$N_2=N_3=1$ & $\R^{N_1\times 1\times 1}$ & $\R^{1\times K\times 1}$ & outer product of two vectors\\
$N_1=N_3=1$ & $\R^{1\times N_2\times 1}$ & $\R^{N_2\times K\times 1}$ &  vector-matrix multiplication\\
$K=N_3=1$ & $\R^{N_1\times N_2\times 1}$ & $\R^{N_2\times 1\times 1}$ &  matrix-vector multiplication (Section \ref{sec:vec})\\
$N_3=1$ & $\R^{N_1\times N_2\times 1}$ & $\R^{N_2\times K\times 1}$ &  matrix-matrix multiplication (Section \ref{sec:matrix})\\ \hline
$N_1=N_2=K=1$ & $\R^{1\times 1\times N_3}$ & $\R^{1\times 1\times N_3}$ & vector circular convolution \eqref{eqn:tprod_p3}\\
$N_1=N_2=1$ & $\R^{1\times 1\times N_3}$ & $\R^{1\times K\times N_3}$ &  vector-matrix circular convolution \eqref{eqn:tprod_r1}\\
$N_2=K=1$ & $\R^{N_1\times 1\times N_3}$ & $\R^{1\times 1\times N_3}$ & matrix-vector circular convolution \eqref{eqn:tprod_r2}\\
$N_1=K=1$ & $\R^{1\times N_2\times N_3}$ & $\R^{N_2\times 1\times N_3}$ & sum of vector circular convolutions \eqref{eqn:tprod_r3}\\
\hline\hline
\end{tabular}
\vspace{0.1in}
\caption{T-Products of Dimension-Reduced Tensors.}
\end{table}

\begin{lemma}\label{lem:tensorinner}
Let $\cA\in\R^{N_1\times N_2\times N_3}, \cB\in\R^{N_1\times K\times N_3}, \cC\in\R^{N_2\times K\times N_3}$, then
\begin{equation}\label{eqn:teninner}
\langle \cA*\cC, \cB\rangle=\langle \cC, \cA^T*\cB\rangle.
\end{equation}

\begin{proof}
\begin{align*}
&\langle \cA*\cC, \cB\rangle=\langle \fold(\circc(\cA)\unfold(\cC)),\cB\rangle=\langle\circc(\cA)\unfold(\cC), \unfold(\cB)\rangle\\
&\quad =\tr\left(\big(\circc(\cA)\unfold(\cC)\big)^T\unfold(\cB)\right)
=\tr\left(\big(\unfold(\cC)\big)^T\circc(\cA^T)\unfold(\cB)\right)\\
&\quad=\langle\unfold(\cC), \circc(\cA^T)\unfold(\cB)\rangle\\
&\quad=\langle \cC,\cA^T*\cB\rangle.
\end{align*}
\end{proof}
\end{lemma}

\begin{lemma}\label{lem:N3}
If $\cA\in\R^{N_1\times N_2\times N_3}$ and $\cX\in\R^{N_2\times K\times N_3}$, then we have
\begin{equation}\label{eqn:lem28}
\|\cA*\cX\|_F\leq \sqrt{N_3}\|\cA\|_F\|\cX\|_F.
\end{equation}

\begin{proof}
$\|\cA*\cX\|_F^2=\|\circc(\cA)\unfold(\cX)\|_F^2\leq\|\circc(\cA)\|_F^2\|\cX\|_F^2=N_3\|\cA\|_F^2\|\cX\|_F^2$.
\end{proof}
\end{lemma}

Due to the circulant matrix multiplication involved in the t-product, the equation $\cA*\cX=\cB$ can be efficiently implemented in the Fourier domain.
\begin{definition}\label{def:F}
Given a tensor $\cA\in\R^{N_1\times N_2\times N_3}$, $F(\cA)$ is the $N_1\times N_2\times N_3$ tensor obtained by taking the 1-dimensional Fourier transform along each tube of $\cA$, i.e.,
$$F(\cA)(i,j,:)=\fft(\cA(i,j,:)),\quad\mbox{for } i\in [N_1], j\in[N_2].$$
\end{definition}

\begin{lemma}\label{lem:fourier}
If $\cA\in\R^{N_1\times N_2\times N_3}$, $\cX\in\R^{N_2\times K\times N_3}$ and $\cB\in\R^{N_1\times K\times N_3}$, then $\cA*\cX=\cB$ is equivalent to
\begin{equation}\label{equ:axbf}
\begin{bmatrix}
F(\cA)_1&&&\\
&F(\cA)_2&&\\
&&\ddots&\\
&&&F(\cA)_{N_3}\end{bmatrix}\begin{bmatrix}F(\cX)_1\\F(\cX)_2\\\vdots\\ F(\cX)_{N_3}\end{bmatrix}
=\begin{bmatrix}F(\cB)_1\\F(\cB)_2\\\vdots\\ F(\cB)_{N_3}\end{bmatrix}.
\end{equation}
\end{lemma}
\begin{proof}
Note that block circulant matrices can be block diagonalized by the Fourier transform. Let $F_{K}$ be the $K\times K$ discrete Fourier transformation matrix, $I_{K}$ the $K\times K$ identity matrix, and $\otimes$ the Kronecker product. Then $\cA*\cX=\cB$ can be converted to
\begin{equation}\label{equ:axbm}
\circc(\cA)\unfold(\cX)=\unfold(\cB).
\end{equation}
By multiplying both sides of \eqref{equ:axbm} on the left with $F_{N_3}\otimes I_{N_1}$, we obtain
\[
\left(F_{N_3}\otimes I_{N_1}\right)\circc(\cA)\left(F_{N_3}^*\otimes I_{N_2}\right)\left(F_{N_3}\otimes I_{N_2}\right) \unfold(\cX)=\left(F_{N_3}\otimes I_{N_1}\right)\unfold(\cB),
\]
which is the same as \eqref{equ:axbf}, considering the fact that $\left(F_{N_3}\otimes I_{N_1}\right)\circc(\cA)\left(F_{N_3}^*\otimes I_{N_2}\right)$ is the block diagonal matrix where each block is the frontal slice of $F(\cA)$.
\end{proof}

Similarly, since $\left(F_{N_3}\otimes I_{N_2}\right)[\circc(\cA)]^*\left(F_{N_3}^*\otimes I_{N_1}\right)$ is a block diagonal matrix where the $j$th block is $[F(\cA)_j]^*$, we have the following lemma.

\begin{lemma}\label{lem:alg1}
\begin{equation}
\left(F_{N_3}\otimes I_{N_2}\right)[\circc(\cA)]^*\unfold(\cX)=
\begin{bmatrix}
F(\cA)_1^*&&&\\
&F(\cA)_2^*&&\\
&&\ddots&\\
&&&F(\cA)_{N_3}^*\end{bmatrix}\begin{bmatrix}F(\cX)_1\\F(\cX)_2\\\vdots\\ F(\cX)_{N_3}\end{bmatrix}.
\end{equation}
\end{lemma}

\section{Tensor Recovery}\label{sec:tenrec}
Let $f$ be an $\alpha_f$-strongly convex function defined on $\R^{N_2\times K\times N_3}$, $\cA\in\R^{N_1\times N_2\times N_3}$ and $\cB\in\R^{N_1\times K\times N_3}$. We assume the linear system $\cA*\cX=\cB$ is consistent and underdetermined. Consider a tensor recovery problem of the following form
\begin{equation}\label{equ:tensormin}
\hat\cX=\argmin_{\cX\in\R^{N_2\times K\times N_3}} f(\cX), \quad\st\quad \cA*\cX=\cB.
\end{equation}
According to the separability of t-product in the first dimension \eqref{eqn:tprod_p1}, the linear constraint $\cA*\cX=\cB$ can be split into $N_1$ reduced ones
\begin{equation}\label{equ:splitcon}
\cA(i,:,:)*\cX=\cB(i,:,:),\quad i\in[N_1].
\end{equation}
For the notational convenience, the $i$th horizontal slice $\cA(i,:,:)$ of $\cA$ will be denoted as $\cA(i)$. One can see that
\begin{equation}\label{equ:ai}
\sum_{i=1}^{N_1}\|\cA(i)*\cX\|_F^2=\|\cA*\cX\|_F^2.
\end{equation}
We further let $H_i:=\{\cX:\cA(i)*\cX=\cB(i)\}$, and $H=\bigcap_{i=1}^{N_1}H_i$ be the feasible set of \eqref{equ:tensormin}. The ``row space'' of $\cA$ is defined as
\[
R(\cA):=\{\cA^T*\cY:\cY\in\R^{N_1\times K\times N_3}\}.
\]
If $N_3=1$, then $R(\cA)$ coincides with the row space of the $N_1\times N_2$ matrix $\cA$.

The optimal solution $\hat\cX$ of \eqref{equ:tensormin} satisfies the following optimality conditions
\begin{equation}\label{equ:tensorminopt}
\cA*\hat\cX = \cB, \quad\partial f(\hat\cX)\cap R(\cA)\neq\emptyset. 
\end{equation}

We propose Algorithm \ref{alg:tensor} that only uses one of the horizontal slice of $\cA$ at each iteration. In order to make sure all horizontal slices are used, we define a control sequence for slice selection at each iteration.
\begin{definition}
A sequence $i:\mathbb{N}\rightarrow[M]$ is called a \textbf{control sequence} for $[M]$ if for any $m\in[M]$, there are infinitely many $k$'s such that $i(k)=m$.
\end{definition}
Control sequences can be defined in a deterministic fashion. For example, a cyclic sequence $\{1,\cdots,N_1$, $1, \cdots, N_1, \cdots\}$ is a control sequence for $[N_1]$.

\begin{algorithm}[htb]
\caption{Regularized Kaczmarz Algorithm for Tensor Recovery}\label{alg:tensor}
\begin{algorithmic}
\State\textbf{Input:} $\cA\in\R^{N_1\times N_2\times N_3}$, $\cB\in\R^{N_1\times K\times N_3}$, control sequence $\{i(k)\}_{k=1}^\infty\subseteq[N_1]$, stepsize $t$, maximum number of iterations $T$, and the tolerance $tol$.
\State\textbf{Output:} an approximate of $\hat \cX$
\State\textbf{Initialize:} $\cZ^{(0)}\in R(\cA)\subset\mathbb{R}^{N_2\times K\times N_3}, \cX^{(0)}=\nabla f^*(\cZ^{(0)})$.
\For{$k=0,1,\ldots,T-1$}
\State $\cZ^{(k+1)}=\cZ^{(k)}+t\cA(i(k))^T*\frac{\cB(i(k))-\cA(i(k))* \cX^{(k)}}{\norm{\cA(i(k))}_F^2}$
\State $\cX^{(k+1)}=\nabla f^*(\cZ^{(k+1)})$
\State If $\norm{\cX^{(k+1)}-\cX^{(k)}}_F/\norm{\cX^{(k)}}_F<tol$, then it stops.
\EndFor
\end{algorithmic}
\end{algorithm}

In Algorithm \ref{alg:tensor}, it can be shown that $\cZ^{(k)}\in\partial f(\cX^{(k)})\cap R(\cA)$ which will be used in our convergence analysis. Regarding the control sequence, one common choice is a cyclic sequence, which is cycling sequentially through the constraints \eqref{equ:splitcon}. Algorithm \ref{alg:tensor} can be viewed as a deterministic algorithm, whereas Algorithm \ref{alg:tensorrand} picks the slices in a random fashion and has gained much attention in recent years~\cite{strohmer2009randomized, needell2014paved, zouzias2013randomized, ma2015convergence}. At each iteration, the probability of picking the $j$th constraint/slice \eqref{equ:splitcon} is proportional to $\|\cA(j)\|_F^2$, which is commonly used in literature \cite{strohmer2009randomized}. Nevertheless, other probability distributions are also allowed; see Corollary \ref{cor:matrix} and Remark \ref{rem:prob}. More discussion on picking appropriate probability distributions can be found in \cite{Dai13, AWL14, CP16}.

\begin{algorithm}[htb]
\caption{Randomized Regularized Kaczmarz Algorithm for Tensor Recovery}\label{alg:tensorrand}
\begin{algorithmic}
\State\textbf{Input:} $\cA\in\R^{N_1\times N_2\times N_3}$, $\cB\in\R^{N_1\times K\times N_3}$, stepsize $t$, maximum number of iterations $T$, and the tolerance $tol$.
\State\textbf{Output:} an approximate of $\hat \cX$
\State\textbf{Initialize:} $\cZ^{(0)}\in R(\cA)\subset\mathbb{R}^{N_2\times K\times N_3}, \cX^{(0)}=\nabla f^*(\cZ^{(0)})$.
\For{$k=0,1,\ldots,T-1$} 
\State pick $i(k)$ randomly from $[N_1]$ with $\Pr(i(k)=j)=\|\cA(j)\|_F^2/\|\cA\|_F^2$,
\State $\cZ^{(k+1)}=\cZ^{(k)}+t\cA(i(k))^T*\frac{\cB(i(k))-\cA(i(k))* \cX^{(k)}}{\norm{\cA(i(k))}_F^2}$
\State $\cX^{(k+1)}=\nabla f^*(\cZ^{(k+1)})$
\State If $\norm{\cX^{(k+1)}-\cX^{(k)}}_F/\norm{\cX^{(k)}}_F<tol$, then it stops.
\EndFor
\end{algorithmic}
\end{algorithm}

\subsection{Convergence Analysis with a Control Sequence}
To analyze the convergence of the proposed Algorithms~\ref{alg:tensor} and \ref{alg:tensorrand}, we state the following proposition which plays an important role in the convergence analysis.

\begin{proposition}\label{thm:tensorD}
Suppose $\cA\in\R^{1\times N_2\times N_3}$, $\cB\in\R^{1\times K\times N_3}$, and $f$ is an $\alpha_f$-strongly convex function defined on $\R^{N_2\times K\times N_3}$. Given an arbitrary $\bar \cZ\in\R^{N_2\times K\times N_3}$ and $\bar \cX=\nabla f^*(\bar \cZ)$, let $\cZ=\bar \cZ+t\frac{\cA^T}{\|\cA\|_F^2}*(\cB- \cA*\bar\cX)$ and $\cX=\nabla f^*( \cZ)$. If $t<\frac{2\alpha_f}{N_3}$, then
\[
D_{f,\cZ}(\cX,  \cH)\leq D_{f,\bar \cZ}(\bar \cX,  \cH)- \frac{t}{\|\cA\|_F^2}\left(1-\frac{t N_3}{2\alpha_f}\right)\|\cB-\cA*\bar\cX\|_F^2,
\]
for any $\cH$ that satisfies $\cA*\cH=\cB$.
\end{proposition}

\begin{proof} For simplicity of notation, let $\cZ=\bar\cZ+s\cW$, where $\cW= \cA^T*(\cB-\cA*\bar\cX)$. By Lemma \ref{lem:N3}, we have $\|\cW\|_F\leq\sqrt{N_3}\|\cA\|_F\|\cB-\cA*\bar\cX\|_F$. By Lemma \ref{lem:tensorinner}, we have
\[
\langle \cW, \cH-\bar\cX\rangle=\langle \cA^T*(\cA*\cH-\cA*\bar\cX), \cH-\bar\cX\rangle=\|\cA*(\cH-\bar\cX)\|_F^2=\|\cB-\cA*\bar\cX\|_F^2.
\]
By \eqref{equ:Df2}, the Bregman distance between $\cX$ and $\cH$ with respect to $f$ and $\cZ$ satisfies
\begin{align*}
D_{f,\cZ}(\cX,\cH)
&=f^*(\cZ)-\langle \cZ,\cH\rangle+f(\cH)\\
&=f^*(\bar\cZ+s \cW)-\langle \bar\cZ+s \cW,\cH\rangle+f(\cH)\\
&\leq f^*(\bar\cZ)+\langle\nabla f^*(\bar\cZ),s\cW\rangle+\frac{1}{2\alpha_f}\|s\cW\|_F^2-\langle \bar\cZ+s \cW,\cH\rangle+f(\cH)\\
&=D_{f,\bar\cZ}(\bar\cX,\cH)-\langle s\cW,\cH\rangle+\langle \bar\cX,s\cW\rangle+\frac{1}{2\alpha_f}\|s\cW\|_F^2\\
&=D_{f,\bar\cZ}(\bar\cX,\cH)-\langle s\cW,\cH-\bar\cX\rangle+\frac{1}{2\alpha_f}\|s\cW\|_F^2\\
&\leq D_{f,\bar\cZ}(\bar\cX,\cH)-s\|\cB-\cA*\bar\cX\|_F^2+\frac{s^2}{2\alpha_f}N_3\|\cA\|_F^2\|\cB-\cA*\bar\cX\|_F^2\\
&=D_{f,\bar\cZ}(\bar\cX,\cH)-s\|\cB-\cA*\bar\cX\|_F\left(1-\frac{sN_3\|\cA\|_F^2}{2\alpha_f}\right).
\end{align*}
The desired result is obtained by setting $s=\frac{t}{\|\cA\|_F^2}$.
\end{proof}

This proposition essentially shows how much the Bregman distance decreases after one iteration. We are now ready to state our first main result.
\begin{theorem}\label{thm:tensorC}
Let $f$ be $\alpha_f$-strongly convex. The sequence generated by Algorithm \ref{alg:tensor} ($t<2\alpha_f/N_3$) satisfies
\begin{equation}\label{equ:tensorDxk}
D_{f,\cZ^{(k+1)}}(\cX^{(k+1)},  \cX)\leq D_{f,\cZ^{(k)}}(\cX^{(k)},  \cX)- t\left(1-\frac{t N_3}{2\alpha_f}\right)\frac{\|\cA(i(k))*(\cX^{(k)}-\cX)\|_F^2}{\|\cA(i(k))\|_F^2},
\end{equation}
for all $\cX\in H_{i(k)}$. Moreover, the sequence $\cX^{(k)}$ converges to the solution of \eqref{equ:tensormin}.
\end{theorem}
\begin{proof}

We apply Proposition \ref{thm:tensorD} where $\cA(i(k))$, $\cB(i(k))$, $\cZ^{(k)}$, $\cZ^{(k+1)}$ are replaced by $\cA$, $\cB$, $\bar\cZ$, $\cZ$, respectively, and then obtain \eqref{equ:tensorDxk}. The rest of the proof is for the convergence of the sequence $\{\cX^{(k)}\}$.

It is known that $\lim\limits_{\|\cZ\|_F\rightarrow\infty}\frac{f^*(\cZ)}{\|\cZ\|_F}=\infty$ (see \cite[Equation (5)]{lorenz2014linearized}). Then for any $\cZ\in\partial f(\cX)$ and arbitrary $\cY$, we have
\[
\lim_{\|\cZ\|_F\rightarrow\infty}\frac{D_{f,\cZ}(\cX,\cY)}{\|\cZ\|_F}=\lim_{\|\cZ\|_F\rightarrow\infty}\frac{f^*(\cZ)-\langle \cZ,\cY\rangle+f(\cY)}{\|\cZ\|_F}\geq \lim_{\|\cZ\|_F\rightarrow\infty}\frac{f^*(\cZ)-\|\cZ\|_F\|\cY\|_F}{\|\cZ\|_F}=\infty.
\]
By \eqref{equ:tensorDxk}, the sequence $\{D_{f,\cZ^{(k)}}(\cX^{(k)},  \hat \cX)\}$ is decreasing, hence bounded and convergent. The above limit implies that $\|\cZ^{(k)}\|_F$ must be bounded. So for a subsequence $\{k_l\}$, we have $\lim\limits_{l\rightarrow\infty}\cZ^{(k_l)}=\tilde \cZ$. Thus, we get
\[
\lim_{l\rightarrow\infty} \cX^{(k_l)}=\lim_{l\rightarrow\infty}\nabla f^*(\cZ^{(k_l)})=\nabla f^*(\tilde \cZ):=\tilde \cX.
\]

We denote the constant $t\left(1-\frac{t N_3}{2\alpha_f}\right)$ by $r$. Since $i(k)$ is a control sequence for $[N_1]$, there exists $j_0\in[N_1]$ such that $i(k_l)$ has infinitely many terms of $j_0$. Without loss of generality, the subsequence can be picked such that
\[
\forall l, i(k_l)\equiv j_0, \text{ and }\{i(k_l), i(k_l+1), \cdots, i(k_{l+1}-1)\}=[N_1].
\]
Therefore, we have
\begin{align}\notag
D_{f,\cZ^{(k_{l+1})}}(\cX^{(k_{l+1})},  \hat \cX)&\leq D_{f,\cZ^{(k_l+1)}}(\cX^{(k_l+1)},  \hat \cX)\\\label{equ:tensorin}
&\leq D_{f,\cZ^{(k_l)}}(\cX^{(k_l)},  \hat \cX)- r\frac{\|\cA(i(k_l))*(\cX^{(k_l)}-\hat\cX)\|_F^2}{\|\cA(i(k_l))\|_F^2}.
\end{align}
By letting $l\rightarrow\infty$, we can get
\[
D_{f,\tilde \cZ}(\tilde \cX,  \hat \cX)
\leq D_{f,\hat \cZ}(\tilde \cX,  \hat \cX)- r\frac{\|\cA(j_0)*(\tilde\cX-\hat\cX)\|_F^2}{\|\cA(j_0)\|_F^2},
\]
which implies that $\cA(j_0)*(\tilde\cX-\hat\cX)=0$ and thereby $\tilde \cX\in H_{j_0}$.

Let $J_{in}:=\{j: \tilde \cX\in H_j\}$. Obviously, $j_0\in J_{in}$ by the previous analysis. Next we use the proof by contradiction to further show  $J_{in}=[N_1]$. Suppose that $J_{out}=[N_1]\backslash J_{in}\neq\emptyset$. For each $l$, we define $n_l$ to be the index in $\{k_l, k_l+1, \cdots, k_{l+1}-1\}$ such that $\{i(k_l), i(k_l+1), \cdots, i(n_l-1)\}\subset J_{in}$ and $i(n_l)\in J_{out}$.

Since $\tilde \cX\in H_j$ for any $j\in\{i(k_l), i(k_l+1), \cdots, i(n_l-1)\}$, we have, by \eqref{equ:tensorDxk},
\[
D_{f,\cZ^{(n_{l})}}(\cX^{(n_{l})},  \tilde \cX)
\leq D_{f,\cZ^{(n_{l}-1)}}(\cX^{(n_{l}-1)},  \tilde \cX)
\leq\ldots\leq D_{f,\cZ^{(k_l)}}(\cX^{(k_l)},  \tilde \cX)
\]
By letting $l\rightarrow\infty$, we have $\lim_{l\rightarrow\infty}\cX^{(n_{l})}=\tilde \cX$.

By choosing a subsequence, we can assume $i(n_l)\equiv j_1\in J_{out}$. By the same analysis as in \eqref{equ:tensorin}, we can get $\tilde \cX\in H_{j_1}$ which is a contradiction. Therefore, we must have $\tilde \cX\in H=\bigcap_{i=1}^{N_1} H_i$.

By \eqref{equ:tensorDxk}, $D_{f,\cZ^{(k)}}(\cX^{(k)}, \tilde\cX)$ is decreasing. Since its subsequence $D_{f,\cZ^{(k_l)}}(\cX^{(k_l)},  \tilde \cX)\rightarrow0$, the entire sequence $D_{f,\cZ^{(k)}}(\cX^{(k)},  \tilde \cX)$ converges to zero as well. This shows $\cX^{(k)}\rightarrow\tilde \cX$ due to \eqref{equ:Df}.

Now we have $\tilde \cZ=\lim_{k\rightarrow\infty}\cZ^{(k)}\in R(\cA)$. Since $\cZ^{(k)}\in\partial f (\cX^{(k)})$, we have $\tilde \cZ\in\partial f(\tilde \cX)$. So $\cA*\tilde\cX=\cB$ and $\tilde \cZ\in \partial f(\tilde \cX)\cup R(\cA)$ fulfills the optimality condition \eqref{equ:tensorminopt}. Since the solution of \eqref{equ:tensormin} is unique, we have $\hat \cX=\tilde \cX$.
\end{proof}

\begin{remark}
We would like to compare Theorem \ref{thm:tensorC} with \cite[Theorem 2.7]{lorenz2014linearized}. Note that \cite[Theorem 2.7]{lorenz2014linearized} considers the split feasibility problem, that is, finding vectors belong to the set $\{\vx|A_i\vx\in Q_i\}$, where $A_i$ are matrices and $Q_i$ are convex sets. Our theorem focuses on a specific minimization problem that recovers tensors.
\end{remark}

\subsection{Convergence Analysis with a Random Sequence}\label{sec:rand}
For the randomized version - Algorithm \ref{alg:tensorrand}, Theorem \ref{thm:tensorC} no longer applies. We will prove a linear convergence rate in expectation if the objective function $f$ satisfies certain conditions.

The dual problem of \eqref{equ:tensormin} is the unconstrained problem
\[
\min_{\cY\in\R^{N_1\times K\times N_3}} g_f(\cY),
\]
where
\begin{equation}\label{equ:dual}
g_f(\cY)=f^*(\cA^T*\cY)-\langle \cY,\cB\rangle.
\end{equation}

\begin{definition}
Let $m:=\displaystyle\min_{\cY}g_f(\cY)$. We will call a function $f$ \emph{admissible} if $g_f$, as defined in \eqref{equ:dual}, is restricted strongly convex on any of $g_f$'s level set $L^{g_f}_\delta:=\{\cY: g_f(\cY)\leq m+\delta\}$. We call $f$ \emph{strongly admissible} if $g_f$ is restricted strongly convex on $\R^{N_1\times K\times N_3}$.
\end{definition}

The following lemma is a consequence of Definition~\ref{def:rsc} and Lemma~\ref{lem:rsc}.
\begin{lemma}\label{lem:tensoressential}
Let $\hat\cX$ be the solution of \eqref{equ:tensormin}.
\begin{enumerate}
\item[(a)] Assume $f$ is admissible. For $\tilde\cX$ and $\tilde\cZ\in\partial f(\tilde\cX)\cap R(\cA)$, there exists $\nu>0$ such that for all $\cX$ and $\cZ\in\partial f(\cX)\cap R(\cA)$ with $D_{f,\cZ}(\cX,\hat \cX)\leq D_{f,\tilde\cZ}(\tilde\cX,\hat \cX)$ it holds
\begin{equation}\label{equ:tensor1}
D_{f,\cZ}(\cX,\hat \cX)\leq\frac{1}{\nu}\|\cA*(\cX-\hat \cX)\|_F^2.
\end{equation}

\item[(b)] If $f$ is strongly admissible, there exists $\nu>0$ such that \eqref{equ:tensor1} holds for all $\cX$ and $\cZ\in\partial f(\cX)\cap R(\cA)$.
\end{enumerate}
\end{lemma}

\begin{proof}
(a) Due to the strong duality, we have that $f(\hat\cX)=-m=-\min_\cY g_f(\cY)$.

Since $\cZ\in R(\cA)$, we let $\cZ=\cA^T*\cY$ for some $\cY$. Then
\begin{align*}
D_{f,\cZ}(\cX,\hat\cX)&=f^*(\cZ)-\langle \cZ,\hat\cX\rangle+f(\hat\cX)=f^*(\cA^T*\cY)-\langle\cA^T*\cY, \hat\cX\rangle+f(\hat\cX)\\
&=f^*(\cA^T*\cY)-\langle\cY, \cA*\hat\cX\rangle+f(\hat\cX)=f^*(\cA^T*\cY)-\langle\cY, \cB\rangle-m\\&=g_f(\cY)-m.
\end{align*}
Similarly, for $\tilde \cX$, there exists $\tilde\cY$ such that  $D_{f,\cZ}(\tilde\cX,\hat\cX)=g_f(\tilde\cY)-m$.

The assumption $D_{f,\cZ}(\cX,\hat \cX)\leq D_{f,\tilde\cZ}(\tilde\cX,\hat \cX)$ implies that $\cY\in\{\cW:g_f(\cW)\leq g_f(\tilde\cY)\} $, a level set of $g_f$. By Lemma \ref{lem:rsc}, the restricted strong convexity of $g_f$ on its level set implies the existence of $\nu>0$ such that
\[
g_f(\cY)-m\leq\frac{1}{\nu}\|\nabla g_f(\cY)\|_F^2, \quad\text{for all }\cY\in\{\cW:g_f(\cW)\leq g_f(\tilde\cY)\}
\]
Furthermore, the gradient of $g_f$ is computed as
$$
\nabla g_f(\cY)=\cA*\nabla f^*(\cA^T*\cY)-\cB=\cA*\nabla f^*(\cZ)-\cB=\cA*\cX-\cB=\cA*(\cX-\hat\cX).
$$
This proves part (a).

(b) The proof is similar to (a) and simpler. Under the assumption that $g_f$ is restricted strong convex on $\R^{N_1\times K\times N_3}$, we no longer need to require the dual variable $\cY$ to be in a level set. The inequality \eqref{equ:tensor1} is thus true for any $\cX$ and $\cZ\in\partial f(\cX)\cap R(\cA)$.
\end{proof}

Lemma~\ref{lem:tensoressential} is a key lemma in proving linear convergence rate for randomized version. However, its assumptions are not easy to check. Below we provide some important examples.

\begin{example}\label{exa:admissible}
Admissible functions. See~\cite{schopfer2016linear}.
\begin{enumerate}
\item[(a)] If $f$ is strongly convex and piecewise quadratic and $\circc(\cA)$ has full row rank, then $f$ is strongly admissible. There are many functions fall under this category, e.g., $f_\lambda(\cX)=\frac{1}{2}\|\cX\|_F^2+\lambda\|\cX\|_1$.
\item[(b)]  Let $f$ be defined on matrices. The objective function of the regularized nuclear norm problem $f(X)=\frac{1}{2}\|X\|_F^2+\lambda\|X\|_*$ is admissible.
\item[(c)] When $f$ is defined on $\R^N$, important examples include $f_\lambda(\vx)=\frac{1}{2}\|\vx\|_2^2+\lambda\|\vx\|_1$, or more generally $\frac{1}{2}\|\vx\|_2^2+\lambda\|U\vx\|_1$, where $U$ is a linear operator.
\end{enumerate}
\end{example}

\begin{remark}\label{rem:gamma}
The constant $\nu$ in Lemma \ref{lem:tensoressential} depends on the tensor $\cA$, the function $f$, and the corresponding level set (if $f$ is only admissible). For a simple example, we let $K=N_3=1$, so $\cA*\cX$ degenerates to the regular matrix-vector multiplication $A\vx$. We let the objective function be $f_1(\vx)=\frac{1}{2}\|\vx\|_2^2$,  we have $\vx=\vz\in R(A)$. Furthermore, the minimizer $\hat \vx$ is the projection of the any feasible vector onto $R(A)$, so $\vx-\hat \vx\in R(A)$. Then $D_{f_1,\vz}(\vx,\hat \vx)=\frac{1}{2}\|\vx-\hat \vx\|_2^2=\frac{1}{2}\|P_{R(A)}(\vx-\hat \vx)\|_2^2\stackrel{\eqref{equ:A}}{\leq}\frac{1}{2\sigma^2_{\min}(A)}\|A(\vx-\hat \vx)\|_2^2$, which means that $\nu=2\sigma^2_{\min}(A)$. For a less trivial example, we refer the readers to \cite[Lemma 7]{LY13} for an explicit computation of $\nu$ for the function $f_\lambda(\vx)=\frac{1}{2}\|\vx\|_2^2+\lambda\|\vx\|_1$ ($A$ does not need to be full row rank). In general, it is hard to quantify $\nu$.
\end{remark}

\begin{theorem}\label{thm:tensorDrate}
Let $f$ be $\alpha_f$-strongly convex and admissible. Let  $\cX^{(k)}, \cZ^{(k)}$ be generated by the Algorithm \ref{alg:tensorrand}. If $t<2\alpha_f/N_3$, then $\cX^{(k)}$ converges linearly to $\hat \cX$ in expectation, i.e.,
\begin{equation}
\E\left[D_{f,\cZ^{(k+1)}}(\cX^{(k+1)},  \hat \cX)\right]\leq \left(1-\frac{\nu t(1-\frac{tN_3}{2\alpha_f})}{ \|\cA\|_F^2}\right)\E\left[D_{f,\cZ^{(k)}}(\cX^{(k)},  \hat \cX)\right],
\end{equation}
and thereby
\begin{equation}\label{equ:thmtensor2}
\E\|\cX^{(k)}-\hat \cX\|_F^2\leq c\left(1-\frac{\nu t(1-\frac{tN_3}{2\alpha_f})}{ \|\cA\|_F^2}\right)^k,
\end{equation}
where $\nu$ is the constant from \eqref{equ:tensor1}.
\end{theorem}
\begin{proof}

Let $\cG^{(k)}:=\cX^{(k)}-\hat \cX$. By Theorem \ref{thm:tensorC},
\begin{equation}\label{equ:tensorD}
D_{f,\cZ^{(k+1)}}(\cX^{(k+1)},  \hat\cX)\leq D_{f,\cZ^{(k)}}(\cX^{(k)},  \hat\cX)- r\frac{\|\cA(i(k))*\cG^{(k)}\|_F^2}{\|\cA(i(k))\|_F^2},
\end{equation}
where $r=t\left(1-\frac{tN_3}{2\alpha_f}\right)$.
We need to analyze the expectation of $\frac{\|\cA(i(k))*\cG^{(k)}\|_F^2}{\|\cA(i(k))\|_F^2}$.

Since $D_{f,\cZ^{(k)}}(\cX^{(k)},\hat \cX)\leq D_{f,\cZ^{(0)}}(\cX^{(0)},\hat \cX)$ for any $k\geq1$, we apply Lemma \ref{lem:tensoressential}(a) and get
\begin{equation}\label{equ:tensor2}
D_{f,\cZ^{(k)}}(\cX^{(k)},\hat \cX)\leq\frac{1}{\nu}\|\cA*\cG^{(k)}\|_F^2.
\end{equation}

At the $k$th iteration, the probability that ${i(k)}=j$ is proportional to $\|\cA(j)\|_F^2$. Therefore, the conditional expectation given the previous iterations $\cX^{(0)}, \cdots, \cX^{(k)}$ is
\begin{align}\label{equ:tensorpi1}
&\E\left[\frac{\|\cA(i(k))*\cG^{(k)}\|_F^2}{\|\cA(i(k))\|_F^2}\bigg|i(1),\dots, i(k-1)\right]
=\sum_{j=1}^{N_1}\frac{\|\cA(j)\|_F^2}{\|\cA\|_F^2} \frac{\|\cA(j)*\cG^{(k)}\|_F^2}{\|\cA(j)\|_F^2}\\
\label{equ:tensorpi3}
\stackrel{\eqref{equ:ai}}{=}&\frac{1}{\|\cA\|_F^2}\|\cA*\cG^{(k)}\|_F^2
\stackrel{\eqref{equ:tensor2}}{\geq}\frac{\nu}{\|\cA\|_F^2}D_{f,\cZ^{(k)}}(\cX^{(k)},\hat \cX).
\end{align}
Finally, we take the expectation of \eqref{equ:tensorD} and get
\[
\begin{aligned}
&\E\left[D_{f,\cZ^{(k+1)}}(\cX^{(k+1)},  \hat \cX)\right]\\
\leq &\, \E\left[D_{f,\cZ^{(k)}}(\cX^{(k)},  \hat\cX)-r\frac{\|\cA(i(k))*\cG^{(k)}\|_F^2}{\|\cA(i(k))\|_F^2}\right]\\
=&\, \E_{i(1),\dots, i(k-1)}\E\left[D_{f,\cZ^{(k)}}(\cX^{(k)},  \hat\cX)- r\frac{\|\cA(i(k))*\cG^{(k)}\|_F^2}{\|\cA(i(k))\|_F^2}\bigg|i(1),\dots, i(k-1)\right]\\
\leq&\, \E_{i(1),\dots, i(k-1)}\left[D_{f,\cZ^{(k)}}(\cX^{(k)},  \hat \cX)-\frac{r\nu}{ \|\cA\|_F^2}D_{f,\cZ^{(k)}}(\cX^{(k)},  \hat \cX)\right]\\
=&\, \left(1-\frac{r\nu}{ \|\cA\|_F^2}\right)\E\left[D_{f,\cZ^{(k)}}(\cX^{(k)},  \hat \cX)\right].
\end{aligned}
\]

To prove \eqref{equ:thmtensor2}, we note that $\E\left[D_{f,\cZ^{(k)}}(\cX^{(k)},  \hat \cX)\right]\leq\left(1-\frac{r\nu}{ \|\cA\|_F^2}\right)^{k} D_{f,\cZ^{(0)}}(\cX^{(0)},  \hat \cX)$.
Then by \eqref{equ:Df}, we have
\begin{align*}\E\|\cX^{(k)}-\hat \cX\|_F^2\leq\frac{2}{\alpha_f}\E\left[D_{f,\cZ^{(k)}}(\cX^{(k)},  \hat \cX)\right]\leq \left[\frac{2}{\alpha_f}D_{f,\cZ^{(0)}}(\cX^{(0)},  \hat \cX)\right]\left(1-\frac{r\nu}{\|\cA\|_F^2}\right)^{k}.
\end{align*}
\end{proof}

So far, we have taken care of consistent and noise-free constraints of the form $\cA*\cX=\cB$. In what follows, we analyze the sensitivity of Algorithm~\ref{alg:tensorrand} by considering the perturbed constraint $\cA*\cX=\tilde\cB$ where $\tilde\cB=\cB+\cE$. Here $\cE$ is typically assumed to be Gaussian noise.

\begin{theorem}\label{thm:tensornoise}
Let $\cA*\cX=\cB$ be the original consistent linear constraint. For the perturbed measurements $\tilde\cB=\cB+\cE$, Algorithm \ref{alg:tensorrand} performs the following updating scheme
\begin{align}
\label{equ:tensornoisez}&\bar \cZ^{(k+1)}=\bar \cZ^{(k)}+t\cA(i(k))^T*\frac{\cB(i(k))+\cE(i(k))-\cA(i(k))* \bar\cX^{(k)}}{\norm{\cA(i(k))}_F^2},\\
\label{equ:tensornoisex}&\bar \cX^{(k+1)}=\nabla f^*(\bar \cZ^{(k+1)}).
\end{align}
Let $f$ be $\alpha_f$-strongly convex and strongly admissible, and the stepsize $t<2\alpha_f/N_3$. If $\displaystyle\epsilon=\max_{i\in [N_1]}\frac{\|\cE(i)\|_F}{\|\cA(i)\|_F}$, then we have
\begin{equation}\label{equ:tensorenoise}
\E\sqrt{D_{f,\bar \cZ^{(k)}}(\bar \cX^{(k)},  \hat \cX)}\leq \left(\sqrt{1-\frac{\nu t(1-\frac{tN_3}{2\alpha_f})}{ \|\cA\|_F^2}}\right)^kD_{f,\bar \cZ^{(0)}}(\bar \cX^{(0)},  \hat \cX)+\frac{\sqrt{2\alpha_f}N_3\|\cA\|_F^2\epsilon}{\nu(1-\frac{tN_3}{2\alpha_f})},
\end{equation}
where $\hat \cX$ is the solution of \eqref{equ:tensormin}.
\end{theorem}

\begin{proof}
Given $\bar \cZ^{(k)}$ and $\bar \cX^{(k)}$, we define
\begin{align}\label{equ:tensornoisezz}
&\cZ^{(k+1)}=\bar \cZ^{(k)}+t\cA(i(k))^T*\frac{\cB(i(k))-\cA(i(k))* \bar\cX^{(k)}}{\norm{\cA(i(k))}_F^2},\\
& \cX^{(k+1)}=\nabla f^*(\cZ^{(k+1)}).
\end{align}
Note that $\cZ^{(k+1)},  \cX^{(k+1)}$ are not those defined in Algorithm \ref{alg:tensorrand}. By \eqref{equ:tensornoisez} and \eqref{equ:tensornoisezz}, we have that $\bar \cZ^{(k+1)}=\cZ^{(k+1)}+\frac{t}{\|\cA(i)\|_F^2}\cA(i)^T*\cE(i)$.

With the update \eqref{equ:tensornoisezz}, we apply Proposition \ref{thm:tensorD}:
\begin{equation}\label{equ:tensorDn}
D_{f,\cZ^{(k+1)}}(\cX^{(k+1)},  \hat \cX)\leq D_{f,\bar \cZ^{(k)}}(\bar \cX^{(k)},  \hat \cX)- t\left(1-\frac{tN_3}{2\alpha_f}\right)\frac{\left\|\cA(i)*(\hat\cX-\bar \cX^{(k)})\right\|_F^2}{\|\cA(i)\|_F^2}.
\end{equation}

Due to the presence of noise, it is not clear if the assumption of Lemma \ref{lem:tensoressential}(a) holds. This is why we require strong admissibility, in which case Lemma \ref{lem:tensoressential}(b) applies and we have
\begin{equation}\label{equ:tensor3}
D_{f,\bar\cZ^{(k)}}(\bar\cX^{(k)},\hat \cX)\leq\frac{1}{\nu}\|\cA*(\bar\cX^{(k)}-\hat\cX)\|_F^2.
\end{equation}
Similar to the analysis in \eqref{equ:tensorpi1}-\eqref{equ:tensorpi3}, we take the conditional expectation of \eqref{equ:tensorDn} and get
\begin{equation}\label{equ:Ec}
\E\left[D_{f,\cZ^{(k+1)}}(\cX^{(k+1)},  \hat \cX)\Big|i(1),\dots, i(k-1)\right]\leq\left(1-\frac{\nu t(1-\frac{tN_3}{2\alpha_f})}{ \|\cA\|_F^2}\right)D_{f,\bar \cZ^{(k)}}(\bar \cX^{(k)},  \hat \cX).
\end{equation}

By Lemma \ref{lem:dd}, we get
\[\begin{aligned}
&\phantom{\leq\leq}D_{f,\bar \cZ^{(k+1)}}(\bar \cX^{(k+1)},  \hat \cX)-D_{f,\cZ^{(k+1)}}(\cX^{(k+1)},  \hat \cX)\\
&\leq \langle \bar \cZ^{(k+1)}-\cZ^{(k+1)}, \bar \cX^{(k+1)}-\hat \cX\rangle
     \leq \|\bar\cZ^{(k+1)}-\cZ^{(k+1)}\|_F\| \bar \cX^{(k+1)}-\hat \cX\|_F\\
&\stackrel{\eqref{equ:Df}}{\leq}\frac{t}{\|\cA(i)\|^2_F}\|\cA(i)^T*\cE(i)\|_F\sqrt{\frac{\alpha_f}{2}}\sqrt{D_{f,\bar\cZ^{(k+1)}}(\bar\cX^{(k+1)},\hat \cX)}\\
&\stackrel{\eqref{eqn:lem28}}{\leq}\frac{tN_3\|\cE(i)\|_F}{\|\cA(i)\|_F}\sqrt{\frac{\alpha_f}{2}}\sqrt{D_{f,\bar\cZ^{(k+1)}}(\bar\cX^{(k+1)},\hat \cX)}\leq tN_3\epsilon\sqrt{\frac{\alpha_f}{2}}\sqrt{D_{f,\bar\cZ^{(k+1)}}(\bar\cX^{(k+1)},\hat \cX)}.
\end{aligned}\]

By denoting $a=tN_3\epsilon\sqrt{\frac{\alpha_f}{2}}, \alpha^2=D_{f,\bar \cZ^{(k+1)}}(\bar \cX^{(k+1)},  \hat \cX), \beta^2=D_{f, \cZ^{(k+1)}}(\bar \cX^{(k+1)},  \hat \cX)$, we can rewrite the above inequality as $\alpha^2-\beta^2\leq a\alpha$, which implies
\[
\alpha\leq \frac{1}{2}(a+\sqrt{a^2+4\beta^2})
\]
and thereby $\alpha\leq a+\beta$. That is,
$$
\sqrt{D_{f,\bar \cZ^{(k+1)}}(\bar \cX^{(k+1)},  \hat \cX)}\leq  a+\sqrt{D_{f,\cZ^{(k+1)}}(\cX^{(k+1)},\hat \cX)}.
$$

For the sake of convenience, we let $\E_c$ be the expectation conditioned on $i(1),\dots, i(k-1)$. Then we have
\begin{align*}
&\E_c\sqrt{D_{f,\bar \cZ^{(k+1)}}(\bar \cX^{(k+1)},  \hat \cX)}\leq a+\E_c\sqrt{D_{f,\cZ^{(k+1)}}(\cX^{(k+1)},\hat \cX)}\\
\leq& a+\sqrt{\E_c\left[D_{f,\cZ^{(k+1)}}(\cX^{(k+1)},  \hat \cX)\right]}
\stackrel{\eqref{equ:Ec}}{\leq}a+\sqrt{\left(1-\frac{\nu t(1-\frac{tN_3}{2\alpha_f})}{ \|\cA\|_F^2}\right)D_{f,\bar \cZ^{(k)}}(\bar \cX^{(k)},  \hat \cX)},
\end{align*}
which implies
\begin{equation}\label{equ:tensorsqrt}
d_{k+1}\leq a+\sqrt{1-\frac{\nu t(1-\frac{tN_3}{2\alpha_f})}{ \|\cA\|_F^2}}d_{k}:=a+\sqrt{1-b}d_k,
\end{equation}
where $d_k:=\E\sqrt{D_{f,\bar \cZ^{(k)}}(\bar \cX^{(k)},  \hat \cX)}$.
After applying \eqref{equ:tensorsqrt} iteratively, we get
\begin{align*}
d_k&\leq(\sqrt{1-b})^kd_0+a\left(1+\sqrt{1-b}+(\sqrt{1-b})^2+\cdots+(\sqrt{1-b})^{k-1}\right)\\
&=(\sqrt{1-b})^kd_0+a\frac{1-(\sqrt{1-b})^k}{1-\sqrt{1-b}}\leq(\sqrt{1-b})^kd_0+a\frac{1}{1-\sqrt{1-b}}\\
&=(\sqrt{1-b})^kd_0+a\frac{1+\sqrt{1-b}}{b}\leq(\sqrt{1-b})^kd_0+\frac{2a}{b},
\end{align*}
which reduces to \eqref{equ:tensorenoise}.
\end{proof}

\section{Special Cases of the Proposed Algorithm}\label{sec:app}

\subsection{Matrix Recovery}\label{sec:matrix}
This section discusses the special cases when $N_3=1$. The tensor $\cA\in\R^{N_1\times N_2\times1}$ degenerates to the $N_1\times N_2$ matrix $A$. Although this has been brought up in Example \ref{exa:admissible}, we would like to include further discussions here.
Specifically, $\cX$ degenerates to the $N_2\times K$ matrix $X$ and $\cB$ degenerates to the $N_1\times K$ matrix $B$. The minimization problem \eqref{equ:tensormin} becomes a matrix recovery problem
\begin{equation}\label{equ:matrixmin}
\min_X f(X) \quad\st\quad AX=B.
\end{equation}

If $f(X)=\lambda\|X\|_*+\frac{1}{2}\|X\|_F^2$, which is an admissible function, we have
$$\nabla f^*(Z)=\prox_{\lambda\|\cdot\|_*}(Z)=D_\lambda(Z).$$
The singular value thresholding operator $D_\lambda(Z)$ is defined as $U\max\{S-\lambda,0\}V^T$, provided that $Z=USV^T$ is the singular value decomposition.

The following corollary is a consequence of Theorem \ref{thm:tensorC} and Theorem \ref{thm:tensorDrate}. Note this function is 1-strongly convex.
Similar to the tensor notation, we will let $A(i)$ be the $i$th row of $A$, and $R_M(A)=\{A^TY:Y\in\R^{N_1\times K}\}$.

\begin{corollary}\label{cor:matrix}
Let $\hat X$ be the solution of \eqref{equ:matrixmin} where $f(X)=\lambda\|X\|_*+\frac{1}{2}\|X\|_F^2$. This function is admissible with the constant $\nu$ (see \eqref{equ:tensor1}). Initializing with $Z^{(0)}\in R_M(A)\subset\mathbb{R}^{N_2\times K}$ and $X^{(0)}=D_\lambda(Z^{(0)})$, we perform the updating scheme
\begin{align}\label{equ:mreg1}
&Z^{(k+1)}=Z^{(k)}+tA(i(k))^T\frac{B(i(k))-A(i(k))X^{(k)}}{\norm{A(i(k))}_F^2},\\
 \label{equ:mreg2}&X^{(k+1)}=D_{\lambda}(Z^{(k+1)}),
 \end{align}
where $\{i(k)\}$ is a slice selection sequence and $t<2$.
\begin{itemize}
\item[(a)] If $\{i(k)\}$ is a control sequence for $[N_1]$, then the sequence $X^{(k)}$ converges to $\hat X$.
\item[(b)] If $i(k)$ is a random sequence such that $\Pr(i(k)=j)=p_j$, then
\begin{equation}
\E\left[D_{f,Z^{(k+1)}}(X^{(k+1)},  \hat X)\right]\leq \left(1-\nu t(1-\frac{t}{2})\min_j\{\frac{p_j}{\|A(j)\|_F^2}\}\right)\E\left[D_{f,Z^{(k)}}(X^{(k)},  \hat X)\right].
\end{equation}
\end{itemize}

\end{corollary}

\begin{remark}\label{rem:prob}
Corollary \ref{cor:matrix}(b) allows a more general probability distribution instead of the specific distribution in Algorithm \ref{alg:tensorrand}. The proof can be easily modified from that in Theorem \ref{thm:tensorDrate}.
\end{remark}
\begin{remark}
It is worth noting that the linear constraint $AX=B$ is in a different form from $\{\langle A_i, X\rangle=b_i, i\in[m]\}$. 
However, the proof of Theorem \ref{thm:tensorDrate} can be easily adapted to this kind of constraints. If we focus on the regularized nuclear norm optimization problem
\begin{equation}\label{equ:rnn}
\min_{X\in\R^{N_2\times K}}\lambda\|X\|_*+\frac{1}{2}\|X\|_F^2\quad \st\quad \langle A_i, X\rangle=b_i, i\in[m],
\end{equation}
it is shown in \cite[Theorem 4.5]{schopfer2019linear} that the following algorithm
\begin{equation}\label{equ:matrix2}
\begin{array}{l}
Z^{(k+1)}=Z^{(k)}+tA_{i(k)}\frac{b(i(k))-\langle A_{i(k)},X^{(k)}\rangle}{\norm{A_{i(k)}}_F^2},\\
X^{(k+1)}=D_{\lambda}(Z^{(k+1)}).
\end{array}
\end{equation}
with a random sequence $i(k)$ has a linear convergence rate in expectation. Related numerical experiments can be found in Section \ref{subsec:inpaint}.
\end{remark}

\subsection{Vector Recovery}\label{sec:vec}
As mentioned in Remark \ref{rem:gamma}, when $\cA\in\R^{N_1\times N_2\times1}$ and $\cX\in\R^{N_2\times 1\times 1}$, $\cA$ degenerates to the $N_1\times N_2$ matrix $A$ and $\cX$ degenerates to the vector $\vx$. In this section, we further discuss this special case in more detail.
Specifically, we consider the following vector version of the minimization problem \eqref{equ:tensormin}
\begin{equation}\label{equ:min}
\hat \vx=\argmin_x f(\vx)\quad\st\quad A\vx=\vb,
\end{equation}
where $R(A)$ in this context is the row space of $A$. As a consequence, Algorithm \ref{alg:tensor} or Algorithm \ref{alg:tensorrand} becomes Algorithm~\ref{alg:vec}.

\begin{algorithm}[htb]
\caption{Regularized Kaczmarz Algorithm for Vector Recovery}\label{alg:vec}
\begin{algorithmic}
\State\textbf{Input:} $A\in\R^{N_1\times N_2}$, $b\in\R^{N_1}$, row selection sequence $\{i(k)\}_{k=1}^\infty\subseteq[N_1]$, stepsize $t$, maximal number of iterations $T$, and tolerance $tol$.
\State\textbf{Output:} an approximate of $\hat \vx$
\State\textbf{Initialize:} $\vz^{(0)}\in R(A)\subset\mathbb{R}^{N_2}, \vx^{(0)}=\nabla f^*(\vz^{(0)})$.
\For{$k=0,1,\ldots,T-1$}
\State $\vz^{(k+1)}=\vz^{(k)}+t\frac{\vb_{i(k)}-A(i(k)) \vx^{(k)}}{\norm{A(i(k))}_2^2}A(i(k))^T$
\State $\vx^{(k+1)}=\nabla f^*(\vz^{(k+1)})$
\State If $\norm{\vx^{(k+1)}-\vx^{(k)}}_2/\norm{\vx^{(k)}}_2<tol$, then it stops.
\EndFor
\end{algorithmic}
\end{algorithm}

In this setting, Theorem \ref{thm:tensorC} and Theorem \ref{thm:tensorDrate} reduce to the following results.
\begin{corollary}\label{cor:vec}
Let $f$ be $\alpha_f$-strongly convex.
\begin{itemize}
\item[(a)] If $i(k)$ is a control sequence for $[N_1]$, then $\vx^{(k)}$ from Algorithm \ref{alg:vec} converges to $\hat\vx$.
\item[(b)] If $f$ is admissible and $i(k)$ is a random sequence such that $\Pr(i(k)=j)=p_j$, then the iterates from Algorithm \ref{alg:vec} satisfy
\begin{equation}
\E\left[D_{f,\vz^{(k+1)}}(\vx^{(k+1)},  \hat \vx)\right]\leq \left(1-\nu t(1-\frac{t}{2\alpha_f})\min_j\{\frac{p_j}{\|A(j)\|_2^2}\}\right)\E\left[D_{f,\vz^{(k)}}(\vx^{(k)},  \hat \vx)\right].
\end{equation}
\end{itemize}
\end{corollary}

Note that Corollary \ref{cor:vec} aligns with the previous results in literature. Part (a) can be found in \cite[Theorem 2.7]{lorenz2014linearized} in a more general setting, and part (b) can be found in \cite[Theorem 4.5]{schopfer2019linear}. The following noisy setting is a new contribution, which is a result of Theorem \ref{thm:tensornoise}.

\begin{corollary}\label{cor:noise}
Let $f$ be $\alpha_f$-strongly convex and strongly admissible.  Let $\bar \vz^{(k)}, \bar\vx^{(k)}$ be generated from Algorithm \ref{alg:vec} ($t<2\alpha_f$) with the noisy constraint $A\vx=\vb+\ve$, i.e.,
\begin{align}
\label{equ:vecnoisez}&\bar \vz^{(k+1)}=\bar \vz^{(k)}+t\frac{\vb_{i(k)}+\ve_{i(k)}-A(i(k)) \bar\vx^{(k)}}{\norm{A(i(k))}_2^2}A(i(k))^T\\
\label{equ:vecnoisex}& \bar \vx^{(k+1)}=\nabla f^*(\bar \vz^{(k+1)}).
\end{align}
Moreover, $\{i(k)\}$ is a random sequence such that $\Pr(i(k)=j)=\frac{\|A(j)\|_2^2}{\|A\|_F^2}$.
Let $\epsilon=\max\limits_{i\in [N_1]}\frac{|\ve_i|}{\|A(i)\|_2}$. We have
\begin{equation}\label{equ:vecnoise}
\E\sqrt{D_{f,\bar \vz^{(k)}}(\bar \vx^{(k)},  \hat \vx)}\leq \left(\sqrt{1-\frac{\nu t(1-\frac{t}{2\alpha_f})}{ \|A\|_F^2}}\right)^kD_{f,\bar \vz^{(0)}}(\bar \vx^{(0)},  \hat \vx)+\frac{\sqrt{2\alpha_f}\|A\|_F^2\epsilon}{\nu(1-\frac{t}{2\alpha_f})},
\end{equation}
where $\hat \vx$ is still the solution of \eqref{equ:min}.
\end{corollary}

In the work \cite{lorenz2014linearized}, some noisy models in practice have been briefly mentioned but without theoretical analysis. Corollary \ref{cor:noise} states that with perturbed measurements, the expected Bregman distance (squared rooted) still enjoys an exponential decay and the iterates are within certain radius (proportional to the noise level) of the true solution.

\subsubsection{Kaczmarz Type of Algorithms}
Important examples that satisfy the assumptions of Corollary \ref{cor:vec} and Corollary \ref{cor:noise} are when $f$ is strongly convex and piecewise quadratic, and $A$ is full row rank (see Example \ref{exa:admissible}).

For $f_1(\vx)=\frac{1}{2}\|\vx\|_2^2$, we have $\vx^{(k)}=\vz^{(k)}$, and Algorithm \ref{alg:vec} becomes the well-known Kaczmarz algorithm~\cite{karczmarz1937angenaherte}
\[
\vx^{(k+1)}=\vx^{(k)}+t\frac{\vb_{i(k)}-A(i(k)) \vx^{(k)}}{\norm{A(i(k))}_2^2}A(i(k))^T,
\]
which is known to converge to the minimum norm solution of $A\vx=\vb$ if the initial $\vx^{(0)}$ is in the row space of $A$. Linear convergence for the randomized Kaczmarz algorithm can be found in \cite{strohmer2009randomized, chen2012almost}.

For $f_\lambda(\vx)=\lambda\|\vx\|_1+\frac{1}{2}\|\vx\|_2^2$, \eqref{equ:min} becomes a regularized version of the Basis Pursuit~\cite{Y08Breg, COS09}:
\begin{equation}\label{equ:rl1}
\hat \vx=\argmin_x \lambda\|\vx\|_1+\frac{1}{2}\|\vx\|_2^2 \quad\st\quad A\vx=\vb,
\end{equation}
and Algorithm \ref{alg:vec} becomes the \emph{sparse Kaczmarz method}
\begin{equation}\label{equ:skm}
\begin{array}{l}\vz^{(k+1)}=\vz^{(k)}+t\frac{\vb_{i(k)}-A(i(k)) \vx^{(k)}}{\norm{A(i(k))}_2^2}A(i(k))^T\\
\vx^{(k+1)}=D_\lambda(\vz^{(k+1)})
\end{array}
\end{equation}
that was proposed in \cite{lorenz2014linearized}.

\subsection{Tensor Nuclear Norm Regularized Minimization}\label{sec:tenrankmin}
In this section, we consider one special case of tensor recovery involving the nuclear-norm regularization and linear measurements. Specifically, we adapt the proposed algorithms for solving the \emph{tensor nuclear norm regularized minimization} problem
\begin{equation}\label{equ:tensorreg}
\hat\cX=\argmin_{\cX\in\R^{N_2\times K\times N_3}} \frac{1}{2}\|\cX\|_F^2+\lambda\|\cX\|_\tnn \quad\st\quad \cA*\cX=\cB,
\end{equation}
where $\cA\in\R^{N_1\times N_2\times N_3}$ and $\cB\in\mathbb{R}^{N_1\times K\times N_3}$.
Here $\norm{\cX}_{tnn}$ is the \emph{tensor nuclear norm} of $\cX$, which is defined through the operator in Definition \ref{def:F} \cite{semerci2014tensor}, that is,
$$\|\cX\|_\tnn:=\sum_{k=1}^{N_3}\|F(\cX)_k\|_*.$$
In addition, tensor nuclear norm can be computed through t-SVD~\cite{KC11}, which involves the SVD of each frontal slice $F(\cX)_k$:
$$F(\cX)_k=\tilde U_k\tilde S_k\tilde V_k^T.$$
Let $\tilde\cU$ be the tensor generated by concatenating $\tilde{U}_k$'s along the third dimension such that $\tilde U_k$ is the $k$th frontal slice of $\tilde\cU$. Likewise, we obtain $\tilde\cS, \tilde\cV$. Denote $\cU=F^{-1}(\tilde\cU)$ and $\cV=F^{-1}(\tilde\cV)$.
The \emph{singular tube thresholding operator} (STT) \cite{semerci2014tensor} is defined as
\[
\sss_\tau(\cX):=\cU*F^{-1}(\max(\tilde \cS-\tau,0))*\cV^T.
\]
As a straightforward application of this definition, we obtain the following lemma.

\begin{lemma}\label{lem:alg2}
For any $\cX\in\R^{N_2\times K\times N_3}$, the equality $\cY=\sss_\lambda(\cX)$ holds if and only if $F(\cY)_k=D_\lambda(F(\cX_k))$ for $k=1,\ldots,N_3$.
\end{lemma}

Next we develop regularized Kaczmarz algorithms for solving \eqref{equ:tensorreg}. It can be shown that $\nabla f^*(\cZ)=\prox_{\lambda\|\cdot\|_\tnn}(\cZ)=\sss_{\lambda}(\cZ)$; see \cite[Theorem 4.2]{lu2019tensor} for more examples.
In this case, Algorithm \ref{alg:tensor} with a given control sequence reduces to Algorithm \ref{alg:tensorreg} which alternates the Kaczmarz step and singular tube thresholding.

\begin{algorithm}[htb]
\caption{Regularized Kaczmarz Algorithm for Solving \eqref{equ:tensorreg}}\label{alg:tensorreg}
\begin{algorithmic}
\State\textbf{Input:} $\cA\in\R^{N_1\times N_2\times N_3}$, $\cB\in\R^{N_1\times K\times N_3}$, control sequence $\{i(k)\}_{k=1}^\infty\subseteq[N_1]$, stepsize $t$, maximum number of iterations $T$, and tolerance $tol$.
\State\textbf{Output:} an approximate of $\hat \cX$
\State\textbf{Initialize:} $\cZ^{(0)}\in R(\cA)\subset\mathbb{R}^{N_2\times K\times N_3}, \cX^{(0)}=\sss_\lambda(\cZ^{(0)})$.
\For{$k=0,1,\ldots,T-1$}
\State $\cZ^{(k+1)}=\cZ^{(k)}+t\cA(i(k))^T*\frac{\cB(i(k))-\cA(i(k))* \cX^{(k)}}{\norm{\cA(i(k))}_F^2}$
\State $\cX^{(k+1)}=\sss_\lambda(\cZ^{(k+1)})$
\State If $\norm{\cX^{(k+1)}-\cX^{(k)}}_F/\norm{\cX^{(k)}}_F<tol$, then it stops.
\EndFor
\end{algorithmic}
\end{algorithm}

The following corollary is a direct consequence of Theorem \ref{thm:tensorC} as the objective function is 1-strongly convex.

\begin{corollary}
The sequence generated by Algorithm \ref{alg:tensorreg} with $t<2/N_3$ satisfies
\begin{equation}
D_{f,\cZ^{(k+1)}}(\cX^{(k+1)},  \cX)\leq D_{f,\cZ^{(k)}}(\cX^{(k)},  \cX)- t\left(1-\frac{t N_3}{2}\right)\frac{\|\cA(i(k))*(\cX^{(k)}-\cX)\|_F^2}{\|\cA(i(k))\|_F^2},
\end{equation}
for all $\cX\in H_{i(k)}$. Moreover, the sequence $\{\cX^{(k)}\}$ converges to the solution of \eqref{equ:tensorreg}.
\end{corollary}

Similarly, we can adapt Algorithm \ref{alg:tensorrand} with a random selection of row index to get a reduced version, i.e., Algorithm \ref{alg:tensorrandreg}, for solving \eqref{equ:tensorreg}.

\begin{algorithm}[htb]
\caption{Randomized Regularized Kaczmarz Algorithm for Solving \eqref{equ:tensorreg}}\label{alg:tensorrandreg}
\begin{algorithmic}
\State\textbf{Input:} $\cA\in\R^{N_1\times N_2\times N_3}$, $\cB\in\R^{N_1\times K\times N_3}$, stepsize $t$.
\State\textbf{Output:} an approximate of $\hat \cX$
\State\textbf{Initialize:} $\cZ^{(0)}\in R(\cA)\subset\mathbb{R}^{N_2\times K\times N_3}, \cX^{(0)}=\sss_\lambda(\cZ^{(0)})$.
\While {termination criteria not satisfied}
\State pick $i(k)$ randomly from $[N_1]$ with $\Pr(i(k)=j)=\|\cA(j)\|_F^2/\|\cA\|_F^2$,
\State $\cZ^{(k+1)}=\cZ^{(k)}+t\cA(i(k))^T*\frac{\cB(i(k))-\cA(i(k))* \cX^{(k)}}{\norm{\cA(i(k))}_F^2}$
\State $\cX^{(k+1)}=\sss_\lambda(\cZ^{(k+1)})$
\EndWhile
\end{algorithmic}
\end{algorithm}

Regarding the convergence analysis of Algorithm~\ref{alg:tensorrandreg}, Theorem \ref{thm:tensorDrate} can not be applied directly as it is not immediately clear if this function is admissible. To address this issue, we propose the following lemma, which can be shown using the definition of Bregman distance and Lemma \ref{lem:alg2}.

\begin{lemma}\label{lem:D}
Let $f(\cX)=\lambda\|\cX\|_{\tnn}+\frac{1}{2}\|\cX\|_F^2$ be defined on $\R^{N_2\times K\times N_3}$, and its reduced version $f_M(X)=\lambda\|X\|_*+\frac{1}{2}\|X\|_F^2$ be defined on $\R^{N_2\times K}$. Then given $z\in\partial f(\cX)$, we have
\[
D_{f,Z}(\cX,\cW)=\sum_{j=1}^{N_3}D_{f_M,F(\cZ)_j}(F(\cX)_j,F(\cW)_j).
\]
\end{lemma}

\begin{theorem}
If $t<2/N_3$, then the sequence generated by Algorithm \ref{alg:tensorrandreg} converges in expectation with
\begin{align}
\E\left[D_{f,\cZ^{(k+1)}}(\cX^{(k+1)},\hat\cX)\right]\leq (1-\frac{\beta\nu}{ \|\cA\|_F^2})\E\left[D_{f,\cZ^{(k)}}(\cX^{(k)},\hat\cX)\right],
\end{align}
where $\beta$ is given by
\begin{equation}\label{eqn:beta}
\beta =\min_{i\in[N_1], j\in[N_3]}t\frac{\|F(\cA(i))_j\|_F^2}{\|F(\cA(i))\|_F^2}(1-\frac{t\|F(\cA(i))_j\|_F^2}{\|F(\cA(i))\|_F^2}).
\end{equation}
\end{theorem}

\begin{proof}
The objective function can be written as
\[
f(\cX)=\lambda\|\cX\|_{\tnn}+\frac{1}{2}\|\cX\|_F^2=\sum_{j=1}^{N_3}\left(\lambda\|F(\cX)_j\|_*+\frac{1}{2}\|F(\cX)_j\|_F^2\right).
\]
In light of Lemma \ref{lem:fourier}, the constraint can be expressed in the Fourier domain as
\[
F(\cA)_jF(\cX)_j=F(\cB)_j,\quad j\in[N_3].
\]
Let $\cV\in\mathbb{R}^{N_2\times K\times N_3}$ and $\cV_k$ be the $k$th frontal slice of $\cV$. Therefore, if
\begin{equation}\label{equ:tensorf}
\hat\cV=\arg\min_\cV\sum_{j=1}^{N_3}\left(\lambda\|\cV_j\|_*+\frac{1}{2}\|\cV_j\|_F^2\right)\quad \st\quad F(\cA)_j\cV_j=F(\cB)_j,\quad j\in[N_3],
\end{equation}
then $F^{-1}(\hat\cV)$ is the minimizer of \eqref{equ:tensorreg}. Note that \eqref{equ:tensorf} can be split into $N_3$ subproblems.

We apply the Fourier transform $F$ to both steps of Algorithm \ref{alg:tensorrandreg}. By Lemma \ref{lem:alg1}, the first step becomes
\[
\begin{aligned}
&\unfold(F(\cZ^{(k+1)}))\\
&\qquad=\unfold(F(\cZ^{(k)}))+\frac{t}{\|\cA(i)\|_F^2}
\diag(F(\cA(i))^*)(F(\cB(i))-\diag(F(\cA(i)))F(\cX^{(k)})),
\end{aligned}
\]
where the block diagonal matrix $\diag(F(\cA(i)))$ and the block matrix $F(\cB(i))$ are defined as
\[
\diag(F(\cA(i)))=\begin{bmatrix}
F(\cA(i))_1&&&\\
&F(\cA(i))_2&&\\
&&\ddots&\\
&&&F(\cA(i))_{N_3}\end{bmatrix},\quad
F(\cB(i))=\begin{bmatrix}F(\cB(i))_1\\F(\cB(i))_2\\\vdots\\ F(\cB(i))_{N_3}\end{bmatrix}.
\]
Likewise, we can define $\diag(F(\cA(i))^*)$ and $F(\cX^{(k)})$. By separating the above updating equation into $N_3$ pieces, we obtain
\begin{equation}\label{equ:f1}
F(\cZ^{(k+1)})_j=F(\cZ^{(k)})_j+\frac{t}{\|\cA(i)\|_F^2}F(\cA(i))_j^*\left(F(\cB(i))_j-F(\cA(i))_jF(\cX^{(k)})_j\right),\quad j\in[N_3].
\end{equation}
By Lemma \ref{lem:alg2}, the second step of Algorithm \ref{alg:tensorrandreg} becomes
\begin{equation}\label{equ:f2}
F(\cX^{(k+1)})_j=D_\lambda(F(\cZ^{(k+1)})_j),\quad j\in[N_3].
\end{equation}
Now we make a change of variables as $F(\cX^{(k)})_j=V_j^{(k)}, F(\cZ^{(k)})_j=W_j^{(k)}$. Then \eqref{equ:f1}-\eqref{equ:f2} become
\begin{align}\label{equ:f11}
&W^{(k+1)}_j=W^{(k)}_j+\frac{t}{\|\cA(i)\|_F^2}F(\cA(i))_j^*\left(F(\cB(i))_j-F(\cA(i))_jV^{(k)}_j\right),\quad j\in[N_3],\\
\label{equ:f22}&V^{(k+1)}_j=D_\lambda(W^{(k+1)}_j),\quad j\in[N_3].
\end{align}
Note that \eqref{equ:f11}-\eqref{equ:f22} are the two major iteration steps for solving the matrix recovery problem
\begin{equation}
\hat V_j=\arg\min_{V_j}\lambda\|V_j\|_*+\frac{1}{2}\|V_j\|_F^2\quad \st\quad F(\cA)_jV_j=F(\cB)_j, j\in[N_3].
\end{equation}
This is in fact the $j$th subproblem of \eqref{equ:tensorf}, so we wish to apply Corollary \ref{cor:matrix}(b). To fit into the format of \eqref{equ:mreg1}, we rewrite
\[
\frac{t}{\|\cA(i)\|_F^2}=\frac{t}{\|F(\cA(i))\|_F^2}=\frac{t\|F(\cA(i))_j\|_F^2}{\|F(\cA(i))\|_F^2}\frac{1}{\|F(\cA(i))_j\|_F^2}.
\]
The stepsize given by
\[
s_j=\frac{t\|F(\cA(i))_j\|_F^2}{\|F(\cA(i))\|_F^2}\leq t\leq 2/N_3\leq2
\]
fits the assumption of Corollary \ref{cor:matrix}. The probability distribution is $p_i=\frac{\|F(\cA(i))\|_F^2}{\|F(\cA)\|_F^2}$. We compute
\[
\frac{p_i}{\|F(\cA(i))_j\|_F^2}=\frac{\|F(\cA(i))\|_F^2}{\|F(\cA)\|_F^2}\frac{1}{\|F(\cA(i))_j\|_F^2}\geq\frac{1}{\|F(\cA)\|_F^2}=\frac{1}{\|\cA\|_F^2}.
\]
So applying Corollary \ref{cor:matrix}(b) yields
\begin{align}
\E\left[D_{f_M,W_j^{(k+1)}}(V_j^{(k+1)},  \hat V_j)\right]\leq \left(1-\frac{\nu s_j(1-\frac{s_j}{2})}{ \|\cA\|_F^2}\right)\E\left[D_{f_M,W_j^{(k)}}(V_j^{(k)},  \hat V_j)\right].
\end{align}
Finally, using Lemma \ref{lem:D}, we have
\[\begin{aligned}
\E\left[D_{f,\cZ^{(k+1)}}(\cX^{(k+1)},\hat\cX)\right]&=\sum_{j=1}^{N_3}\E\left[D_{f_M,W^{(k+1)}_j}(V_j^{(k+1)},  \hat V_j)\right]\\
&\leq \sum_{j=1}^{N_3}\left(1-\frac{\nu s_j(1-\frac{ s_j}{2})}{ \|\cA\|_F^2}\right)\E\left[D_{f_M,W_j^{(k)}}(V_j^{(k)},  \hat V_j)\right]\\
&\leq (1-\frac{\nu\beta}{ \|\cA\|_F^2})\E\left[D_{f,\cZ^{(k)}}(\cX^{(k)},\hat\cX)\right],
\end{aligned}\]
where
\[
\beta=\min_{i\in[N_1],j\in[N_3]}s_j(1-s_j/2)
=\min_{i\in[N_1],j\in[N_3]}t\frac{\|F(\cA(i))_j\|_F^2}{\|F(\cA(i))\|_F^2}(1-\frac{t\|F(\cA(i))_j\|_F^2}{\|F(\cA(i))\|_F^2})>0.
\]
\end{proof}

\section{Numerical Experiments}\label{sec:exp}
In this section, we illustrate the performance of the proposed algorithms in several application problems, including one-dimensional sparse signal recovery, low-rank image inpainting, low-rank tensor recovery, and image deblurring. There are two special cases of our proposed algorithms when the constraint selection sequence is either cyclic or random. For the random version, we take the average of all the results obtained by running 50 trials. To save the computational time, we only execute the deterministic version with a cyclic control sequence for image inpainting and deblurring tests. For image deblurring, we also consider a batched version of the proposed algorithm to improve the performance. However, the batching technique has shown very limited performance enhancement in our other experiments, so we skip those experiments.

To make fair performance comparisons, we adopt the widely used quantitative metrics. For tensor recovery, we use the relative error (RelErr) defined as
\[
\text{RelErr}=\frac{\norm{\cX-\hat\cX}_F}{\norm{\hat\cX}_F},
\]
where $\cX$ is the estimation of the ground truth tensor $\hat\cX$. As one of the most important image quality metrics, the peak signal-to-noise ratio (PSNR) is defined as
\[
\text{PSNR}=20\log(I_{\max}/\|\hat X-X\|_F),
\]
where $X$ is the estimation of the noise-free image $\hat X$ and $I_{\max}$ is the maximum possible image intensity.
In addition, the structural similarity index  (SSIM) between two images $X$ and $Y$ is defined as
\[
\text{SSIM}=\frac{(2\mu_x\mu_y+C_1)(2\sigma_{xy}+C_2)}{(\mu_x^2+\mu_y^2+C_1)(\sigma_x^2+\sigma_y^2+C_2)},
\]
where $\mu_x,\sigma_x$ are the mean and standard deviation of the image $X$, $\sigma_{xy}$ is the cross-covariance between $X$ and $Y$, and $C_1$ and $C_2$ are the luminance and contrast constants. Both PSNR and SSIM values can be obtained efficiently in MATLAB via \verb"PSNR" and \verb"SSIM" respectively.

All the numerical experiments are implemented using MATLAB R2019a for Windows 10 on a desktop PC with 64GB RAM and a 3.10GHz Intel Core i9-9960X CPU.

\subsection{One-dimensional Sparse Signal Recovery}
We will compare the performance of the following three methods for solving \eqref{equ:rl1}: (1) linearized Bregman iteration~\cite{Y08Breg, COS09} (denoted by LinBreg); (2) alternating direction method of multipliers (ADMM)~\cite{admm}; (3) our proposed regularized Kaczmarz method \eqref{equ:skm} with random or deterministic cyclic  sequence (denoted by RK-rand and RK-cyc). In particular, LinBreg has the following iterations
\begin{equation}\label{equ:lb}
\begin{aligned}
\vz^{(k+1)}&=\vz^{(k)}+tA^T(\vb-A \vx^{(k)}),\\
\vx^{(k+1)}&=S_\lambda(\vz^{(k+1)}).
\end{aligned}
\end{equation}
For ADMM, we rewrite \eqref{equ:rl1} as
\[
\min_{\vx,\vw}\frac{1}{2}\|\vx\|_2^2+\lambda\|\vw\|_1,\quad\st\quad A\vx=\vb,\, \vw=\vx,
\]
and the corresponding augmented Lagrangian reads as
\[
L(\vx,\vw,\vu_1,\vu_2)=\frac{1}{2}\|\vx\|_2^2+\lambda\|\vw\|_1+\frac{\rho_1}{2}\|A\vx-\vb+\vu_1\|_2^2+\frac{\rho_2}{2}\|\vx-\vw+\vu_2\|_2^2.
\]

In our experiment, $A$ is a $200\times1000$ Gaussian matrix. The ground truth vector $\hat \vx$ is a 10-sparse vector whose support is randomly generated and the entries on each support index are independent and follow the normal distribution with mean 1 and standard deviation 1.  The parameters are all tuned to achieve optimal performance. For the regularized Kaczmarz method, the step size is $t=40$, and we show the results when the indices are chosen cyclically or randomly. For the linearized Bregman iteration, the step size $t=20$. For ADMM, we pick $\rho_1=10$ and $\rho_2=100$.

Figure~\ref{fig:1dten} shows the relative error of all four methods versus the running time. Both versions of the regularized Kaczmarz algorithms are outperforming LinBreg and ADMM. Moreover, the cyclic version of the regularized Kaczmarz method performs slightly better than the randomized one.

\begin{figure}
\centering
\includegraphics[width=.5\textwidth]{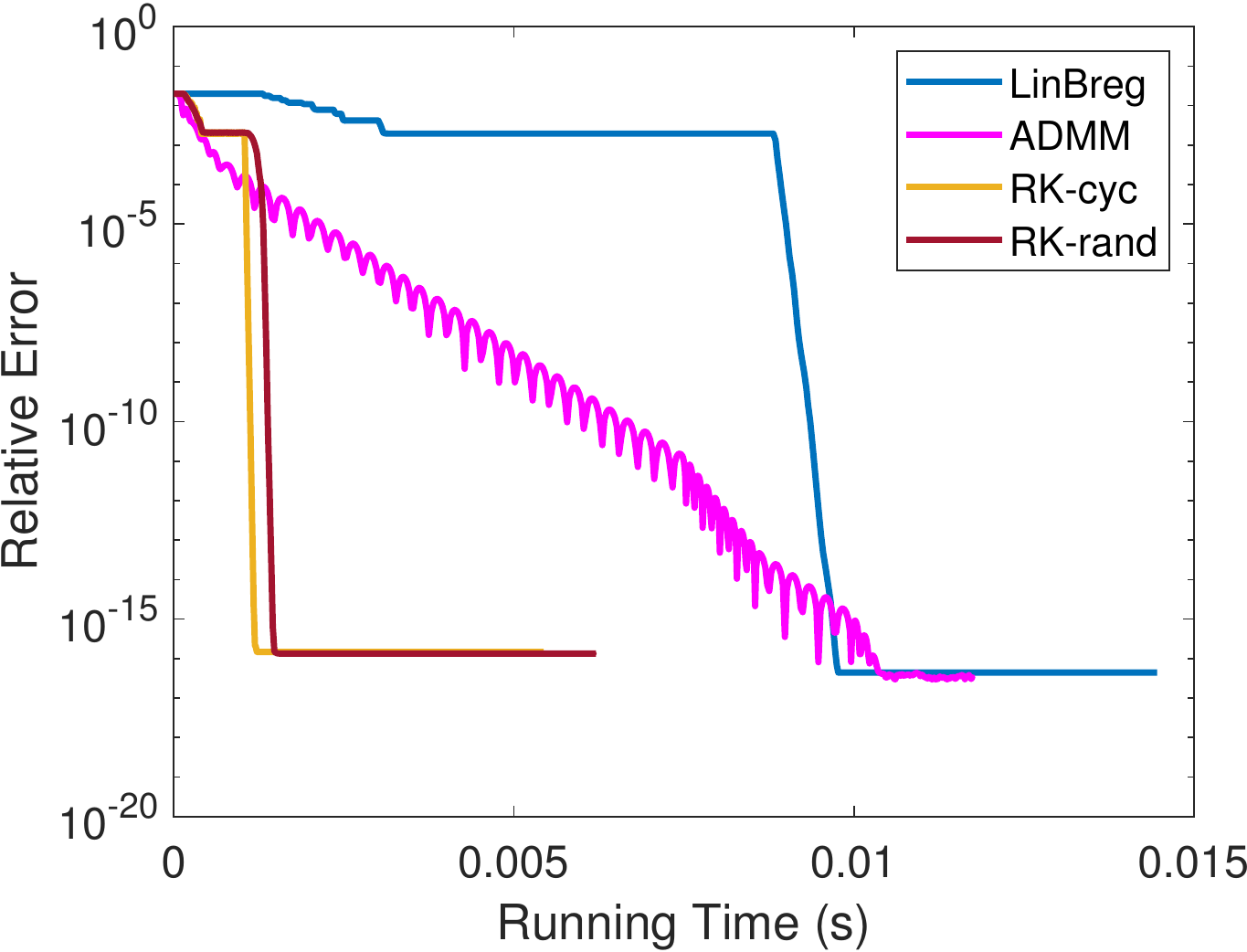}
\caption{One-dimensional sparse signal recovery.}\label{fig:1dten}
\end{figure}

\subsection{Image Inpainting}\label{subsec:inpaint}
In the second experiment, we consider a low-rank image inpainting problem. The test image is a checker board image of size $128\times128$ with a large missing rectangular area; see the first image of Figure \ref{fig:checker}. This image can be described by a rank-two matrix, so it is appropriate to be recovered via the model \eqref{equ:rnn}.  Let $I$ be the image to be recovered, and $\Omega$ be the set of indices whose pixel values are known. Then the linear constraints are $\{\langle P_{ij}, X\rangle=I_{ij}, (i,j)\in\Omega\}$, where $P_{ij}$ is the matrix whose entries are all zeros except its $ij$th entry is one. We use $X_\Omega$ for keeping the pixel intensities in $\Omega$ and setting other pixel intensities to be zero.

We consider three image inpainting methods: (1) total variation (TV) based image inpainting \cite{shen2002mathematical,getreuer2012total}, (2) linearized Bregman iteration \cite{CCZ10} (denoted by LinBreg as before), (3) the proposed regularized Kaczmarz method~\ref{equ:matrix2}. For LinBreg, we adopt the following updating scheme in analogy to \eqref{equ:lb}
\begin{equation}\label{equ:matrix3}
\begin{array}{l}
Z^{(k+1)}=Z^{(k)}+t(I-X^{(k)})_\Omega,\\
X^{(k+1)}=D_{\lambda}(Z^{(k+1)}).
\end{array}
\end{equation}
The second is the regularized Kaczmarz method \eqref{equ:matrix2}. In this application, the first step of \eqref{equ:matrix2} becomes $Z^{(k+1)}=Z^{(k)}+t(I-X^{(k)})_{(i,j)}$, which means only one pixel from $\Omega$ is updated. Again, the choice of indexing can be cyclic or random. In our numerical experiment, we will use a more general version  $Z^{(k+1)}=Z^{(k)}+\sum_{(i,j)\in T(k)}t(I-X^{(k)})_{ij}=Z^{(k)}+t(I-X^{(k)})_{T(k)}$, where $T(k)\subset\Omega$, so that $|T(k)|$ many pixels are updated in one iteration. A different index set $T(k)$ is chosen at each iteration $k$, but we do require that they have the same \emph{batch size}, i.e., the cardinality $|T(k)|=b$. 
Therefore, our algorithm for this specific case reads as
\begin{equation}\label{equ:matrix4}
\begin{array}{l}
Z^{(k+1)}=Z^{(k)}+t(I-X^{(k)})_{T(k)},\\
X^{(k+1)}=D_{\lambda}(Z^{(k+1)}).
\end{array}
\end{equation}
In the methods involving SVT, we choose $\lambda=1500$.

\begin{figure}[htb]
\centering\setlength{\tabcolsep}{2pt}
\begin{tabular}{cccc}
\includegraphics[width=.24\textwidth]{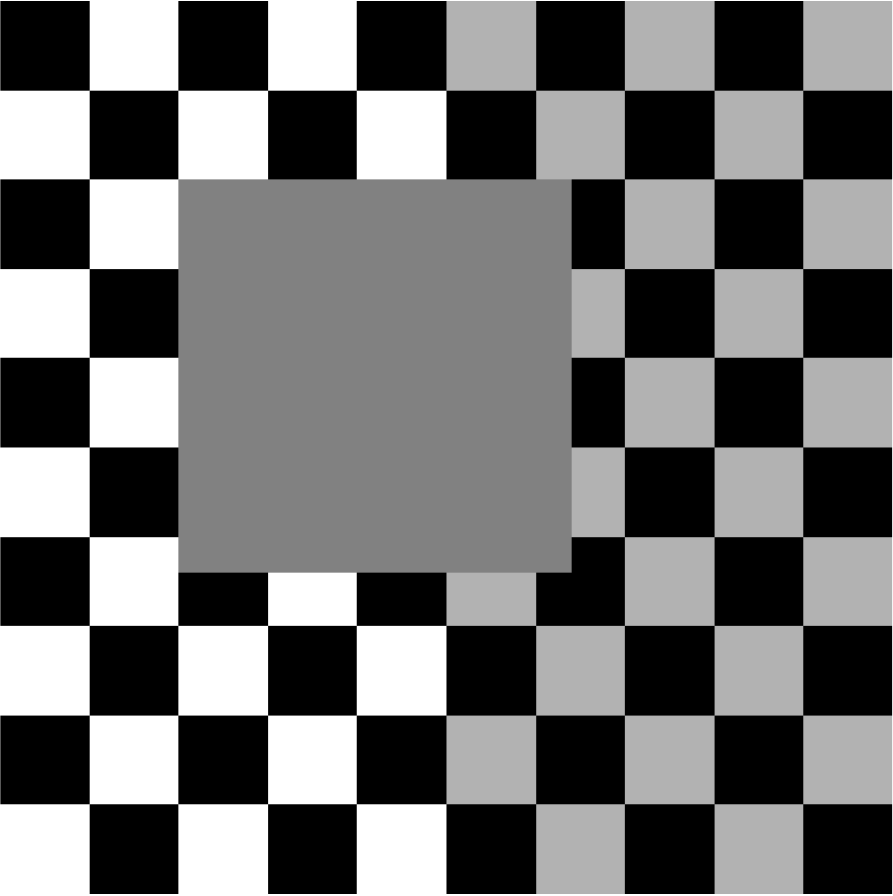}&
\includegraphics[width=.24\textwidth]{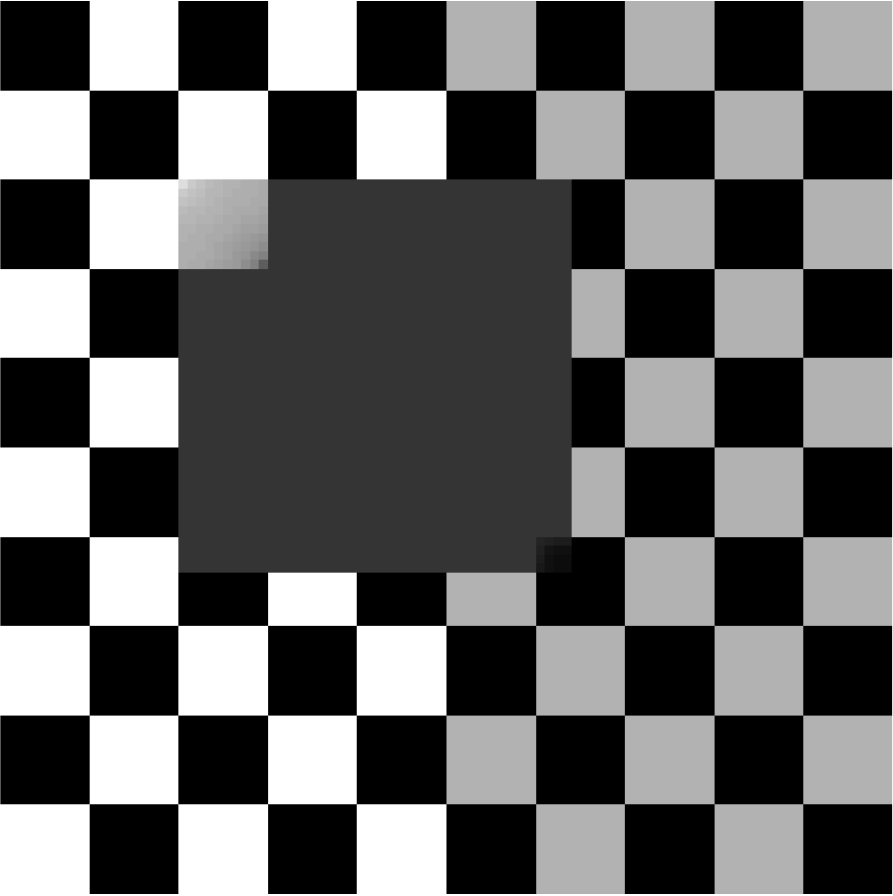}&
\includegraphics[width=.24\textwidth]{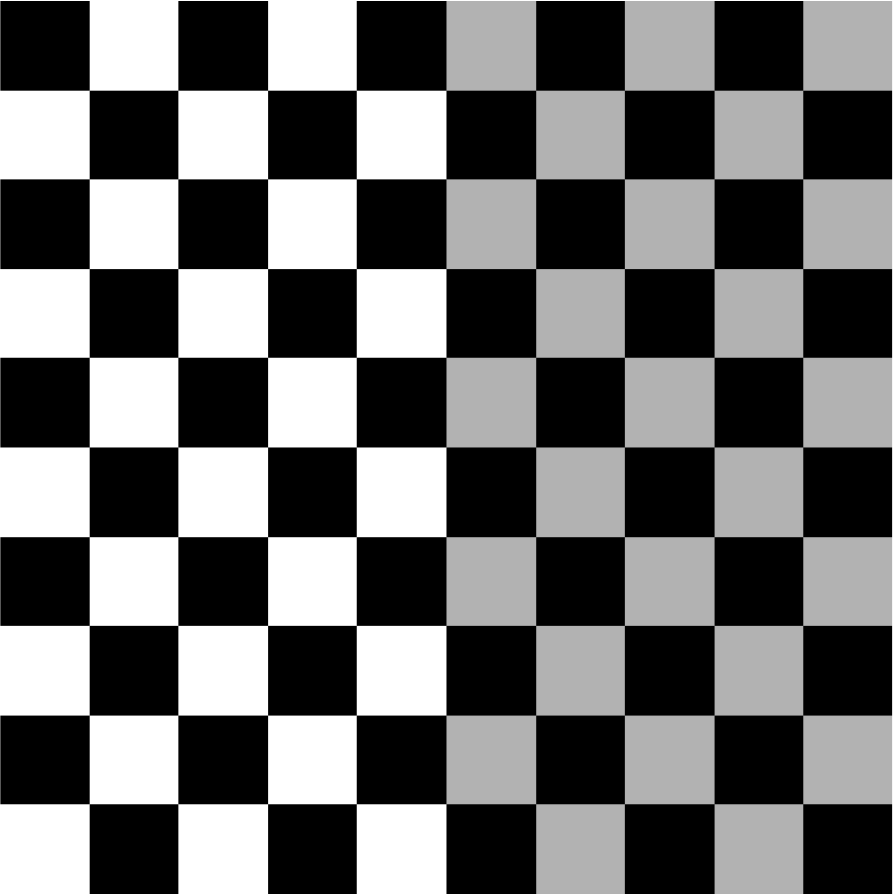}&
\includegraphics[width=.24\textwidth]{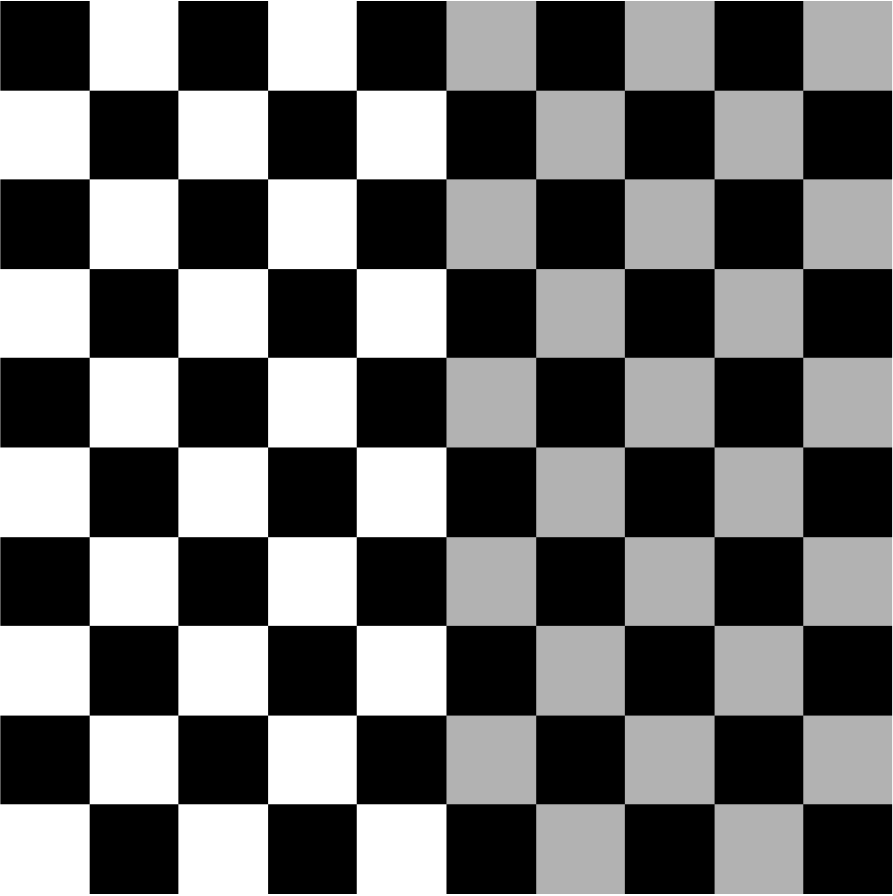}\\
Original&TV   &LinBreg   &Proposed \\
        &fail&PSNR=40.96&PSNR=76.17
\end{tabular}
\caption{Image inpainting without noise. The original checker board image has a missing box. The linearized Bregman result uses \eqref{equ:matrix3} with $t=1$. Our result uses \eqref{equ:matrix4} with $t=9$ and batch size 2000. The running times are: 1.24s (TV), 0.95s (LinBreg), 0.60s (proposed).}\label{fig:checker}
\end{figure}

Figure~\ref{fig:checker} compares all the listed methods quantitatively and qualitatively. For the TV result, the image was recovered by minimizing the functional $\lambda\|\nabla X_x\|_1+\|\nabla X_y\|_1+\frac{1}{2}\|(X-I)_\Omega\|_2^2$. The TV regularization only considers the piecewise constant type of smoothness and thus it may not handle texture-like images with a low-rank structure very well. Since the missing area is relatively large in the test image, TV fails to recover the image as expected. As shown in the last two sub-figures of Figure \ref{fig:checker}, either the linearized Bregman iteration, or the regularized Kaczmarz iteration is able to achieve almost perfect reconstruction. In our regularized Kaczmarz iteration, we choose batch size to be 2000 with a cyclic indexing. The regularized Kaczmarz enjoys faster convergence.

The batch size $b$ plays an important role in a lot of optimization algorithms, e.g., stochastic gradient descent. In this application, $b$ can be viewed as the number of pixels updated in step 1 of \eqref{equ:matrix4}. If $b=|\Omega|$, then our algorithm coincides with \eqref{equ:matrix3}. This is related to Stochastic Gradient Matching Pursuit (StoGradMP)~\cite{needell2016stochastic} especially with a random choice of the constraints. Some other relevant works include the block Kaczmarz algorithm~\cite{needell2014paved, CP16}. The proof presented in this paper can also be adapted to show that the following algorithm
\begin{equation}\label{equ:matrixb}
\left\{
\begin{aligned}
&\text{Pick an index set $T(k)$ whose cardinality is $b$,} \\
&Z^{(k+1)}=Z^{(k)}+t\left(P_{span(A_i, i\in T(k))}\hat X-P_{span(A_i, i\in T(k))}X^{(k)}\right),\\
&X^{(k+1)}=D_{\lambda}(Z^{(k+1)})
\end{aligned}
\right.
\end{equation}
produces a sequence $\{X^{(k)}\}$ that converges to the solution of \eqref{equ:rnn}. Note that in the image inpainting problem, \eqref{equ:matrix4} is exactly this block version \eqref{equ:matrixb} due to the orthogonality of the indicator matrices $P_{ij}$. The batch size does influence the convergence of \eqref{equ:matrix4}. For this particular example, we have $|\Omega|=8064$ and we found a batch size around 2000 to be optimal.

\subsection{Low-Rank Tensor Recovery}
Assume that the ground truth tensor $\cX\in\R^{N_2\times K\times N_3}$ with a small tubal rank satisfies the tensor system
\[
\cA*\cX=\cB,\quad\cA\in\R^{N_1\times N_2\times N_3}.
\]
Then we consider the tensor recovery model \eqref{equ:tensorreg} with the tensor nuclear norm regularization. It can be empirically shown that the proposed algorithm~\ref{alg:tensor} can achieve good performance in terms of accuracy and convergence speed when $N_1$ is larger than the other dimensions $N_2,N_3,K$ but may fail to converge in other scenarios. As an illustration, we show the performance of our algorithm with cyclic and random control sequences in Figure~\ref{fig:tensys} where $N_1=200$, $N_2=N_3=K=100$ and the maximal number of iterations as 2000. Both the coefficient tensor $\cA$ and the ground truth tensor $\cX$ are normally distributed, and the tubal rank of $\cX$ is two by taking the hard thresholding of singular tubes after t-SVD. For the random case, we take the average of 50 trials. One can see that the cyclic control sequence achieves faster convergence with much smaller error but with slightly more running time.

\begin{figure}
\centering
\includegraphics[width=.5\textwidth]{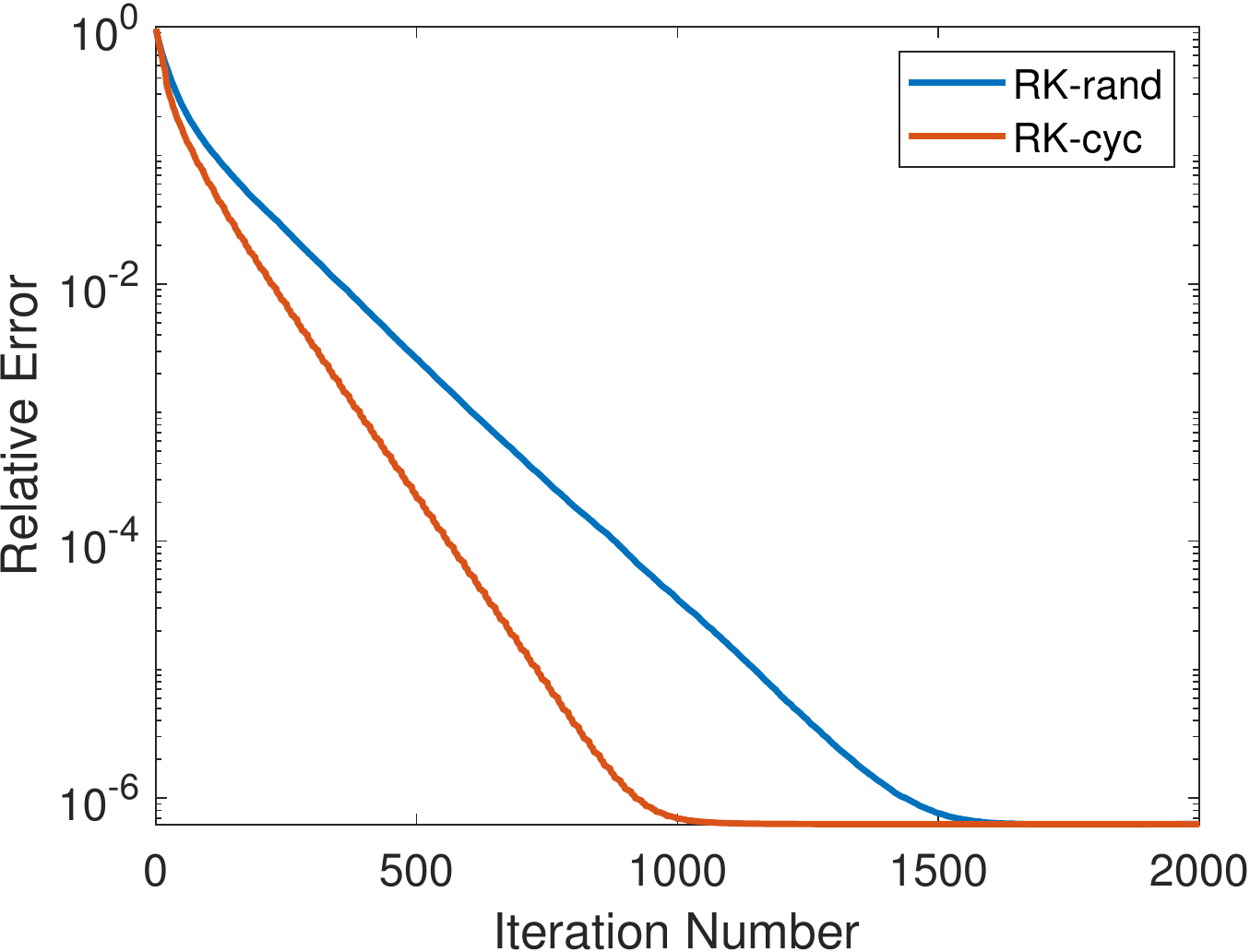}
\caption{Low-rank tensor recovery. Running times for the random and cyclic control sequences are: 0.8721 s and 0.8612 s, respectively. Both run 2000 iterations.}\label{fig:tensys}
\end{figure}

\subsection{Tensor-Based Image Deconvolution}\label{subsec:decov}
Consider an image $\widehat{X}\in\mathbb{R}^{m_1\times n_1}$ that is degraded by taking the convolution with a point spread function $\widehat{H}\in\mathbb{R}^{m_2\times n_2}$. Let $m=m_1+m_2-1$ and $n=n_1+n_2-1$. Using the zero padding, both $\widehat{X}$ and $\widehat{H}$ can be extended to two respective matrices $X,H$ of size $m\times n$. By the construction, we have
\[
H\circledast X=B,
\]
where $\circledast$ is the 2D convolution. Next we establish the equivalence between 2D convolution and t-product by creating a doubly block circulant matrix \cite[Appendix]{abidi2016optimization}.
Let $h_i$ be the $i$th column of $H$, and $A_i:=\mathrm{circ}(h_i)\in\R^{n\times n}, i\in[m]$ be the circulant matrix generated by $h_i$. Let $A_i$ be the $i$th frontal slice of $\cA\in\R^{n\times n\times m}$. Then we create a doubly block circulant matrix $\circc(\cA)$ based on $\cA$. Let $\cX\in\mathbb{R}^{n\times 1\times m}$ be the tensor version of $X$ by setting $\cX(j,1,i) = X(i,j)$ for $i\in[m]$ and $j\in[n]$. Then the 2D convolution can be represented as
\[
\vect(H\circledast X)=\circc(\cH)\vect(X)=\cA*\cX.
\]
Here $\vect(\cdot)$ is a vectorization operator by row-wise stacking. To recover $\cX$, we consider the low-rank tensor recovery model
\[
\min_{\cX\in\R^{n\times 1\times m}}\lambda\norm{\cX}_*\quad \st\quad \cA*\cX=\cY
\]
where $\cY\in\R^{n\times 1\times m}$ is the tensor version of the observed blurry image $Y\in\R^{m\times n}$. Note that the tubal rank of the ground truth $\cX$ is at most one since the second dimension of $\cX$ is one, which implies that $\cX$ is low-rank. The conversion between $Y$ and $\cY$ is the same as that between $X$ and $\cX$. Next we apply Algorithm \ref{alg:tensorrandreg} to recover $\cX$ and thereby the image $X$.

We test an image ``house'' of size $256\times 256$, which is degraded by a Gaussian convolution kernel of size $9\times 9$ with standard deviation 2. Then we compare various image deblurring methods, including TV image deblurring \cite{chan2011augmented}, nonlocal means (NLM) \cite{buades2011non}, BM3D \cite{dabov2006image}, and our algorithm with batch sizes $b=20,40,60,80$ and $\alpha=1$ $\lambda=0.1$. Here TV, NLM and BM3D are performed in the plug-and-play ADMM image recovery framework \cite{chan2016plug}. The MATLAB sources codes can be found in \url{https://www.mathworks.com/matlabcentral/fileexchange/60641-plug-and-play-admm-for-image-restoration}. To avoid ringing artifacts, we extend the image symmetrically along the boundary $14$ pixels, apply the algorithm and cut the results back to the normal size. Figure~\ref{fig:tendeblur} show the images recovered by TV, NLM, BM3D and our algorithm with $b=80$. All results are compared quantitatively in Table~\ref{tab:tendeblur} in terms of PSNR, SSIM and running time.
\begin{figure}
\centering
\begin{tabular}{ccc}
\includegraphics[width=.3\textwidth]{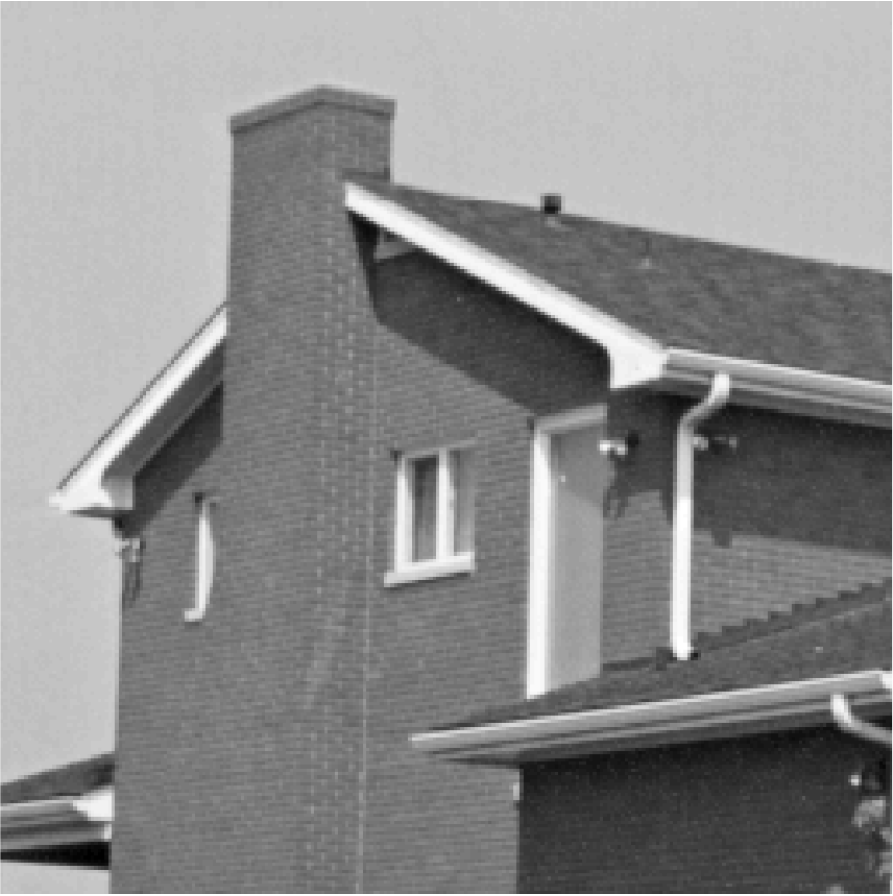}&
\includegraphics[width=.3\textwidth]{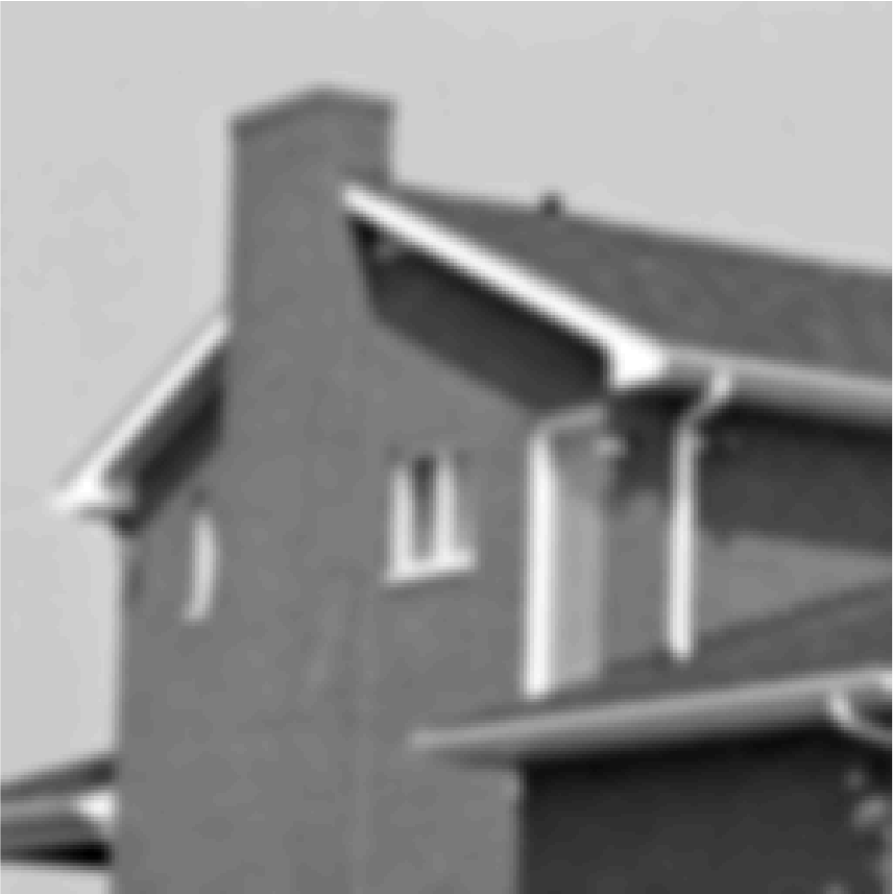}&
\includegraphics[width=.3\textwidth]{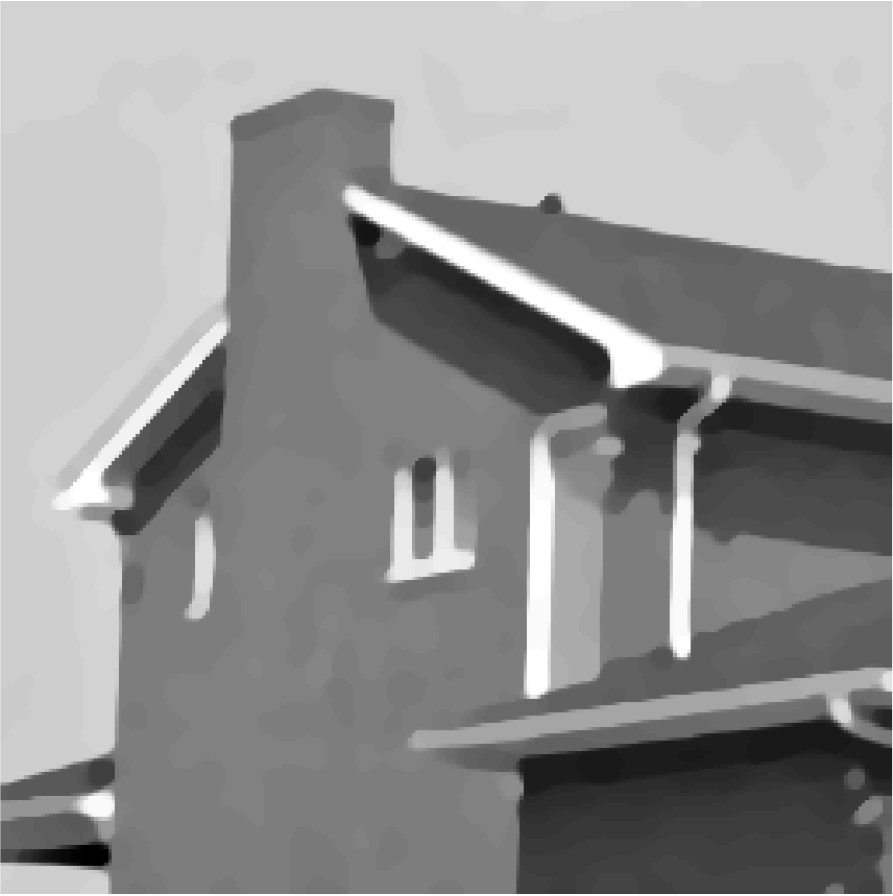}\\
Clean & Blurry & TV\\
\includegraphics[width=.3\textwidth]{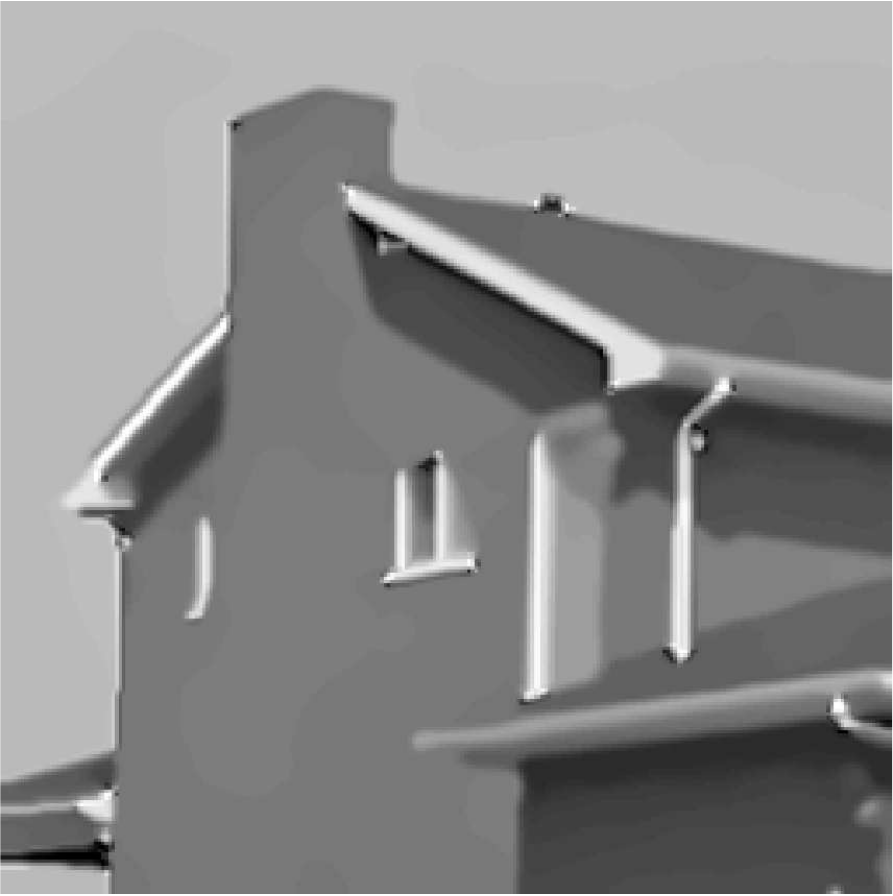}&
\includegraphics[width=.3\textwidth]{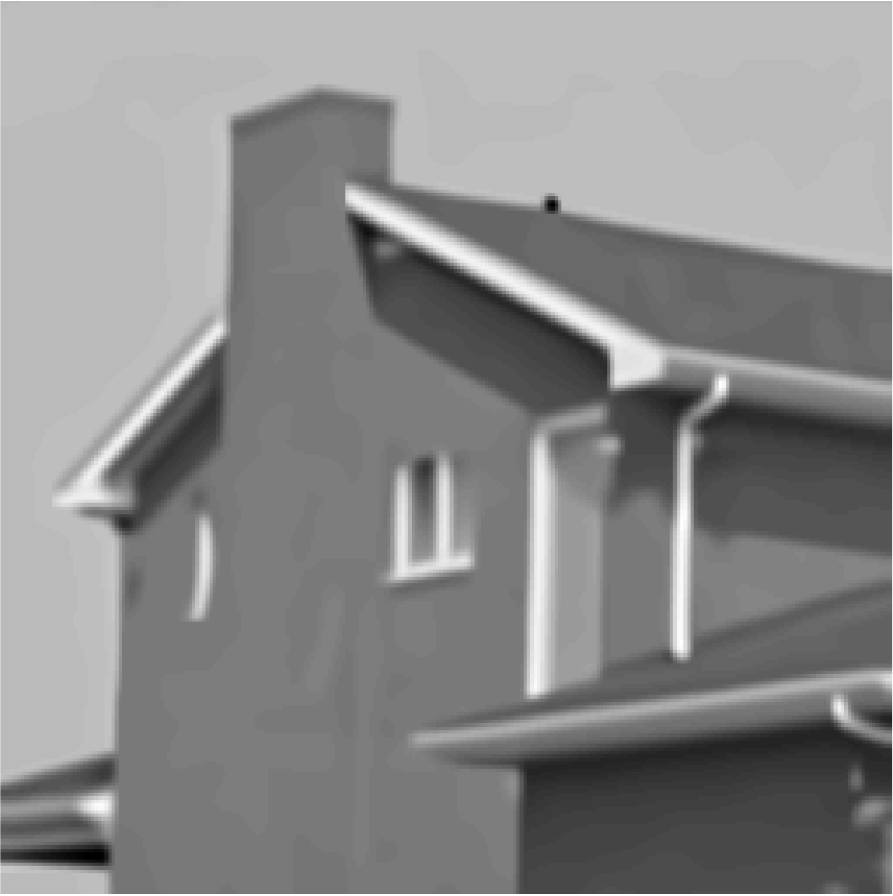}&
\includegraphics[width=.3\textwidth]{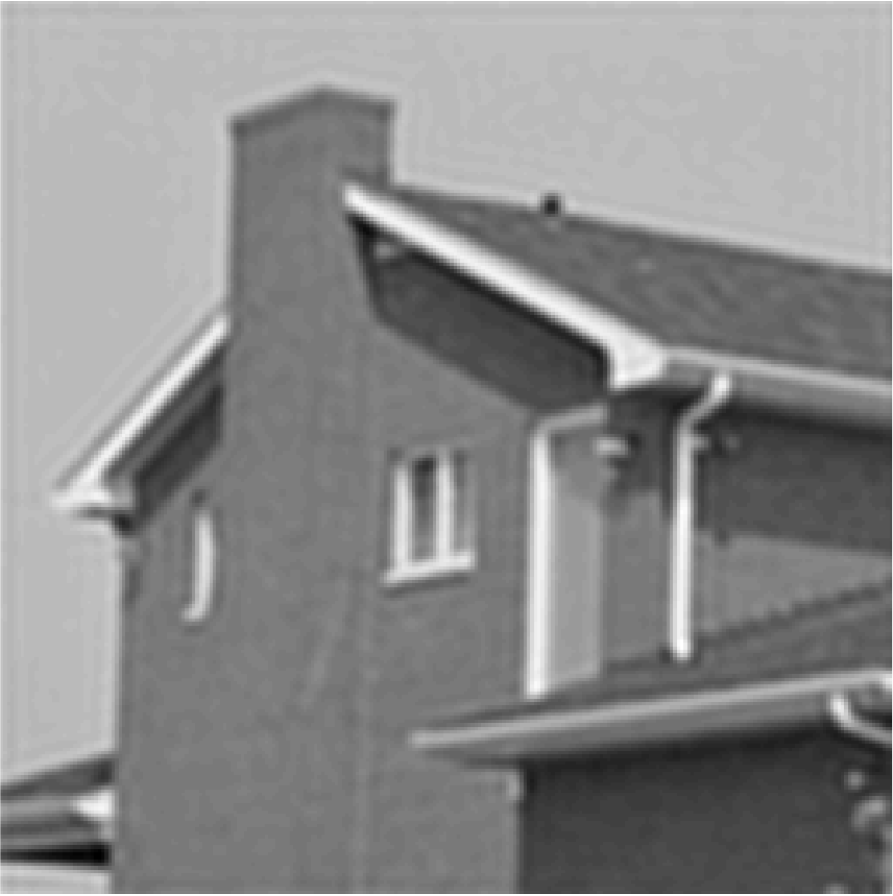}\\
NLM & BM3D & Proposed
\end{tabular}
\caption{Visual comparison of various image deblurring methods, including TV, NLM, BM3D and our proposed tensor-based deblurring algorithm with sequential control sequence and blocksize $b=80$.}\label{fig:tendeblur}
\end{figure}

\begin{table}
\centering
\begin{tabular}{c|c|c|c}
\hline \hline
Method & PSNR & SSIM & Running Time (s) \\ \hline
TV & 29.48 & 0.8180 & 33.34\\
NLM & 29.16 & 0.8113 & 29.27\\
BM3D & 30.69 & 0.8381 & 318.58\\
$b=20$ & 30.90 & 0.8257 & 53.34\\
$b=40$ & 31.10 & 0.8451& 72.88\\
$b=60$ & 31.11 & 0.8457 & 97.96\\
$b=80$ & 31.12 & 0.8461 & 122.13\\
\hline\hline
\end{tabular}
\caption{Quantitative comparison of various image deblurring methods. Rows 5-8: the proposed tensor-based image deblurring algorithm with sequential control sequence and batch size $b=20,40,60,80$. }\label{tab:tendeblur}
\end{table}

More generally, we consider an image sequence $\cX\in\mathbb{R}^{n\times p\times m}$ with $p$ frames (a video). Assume that all frames are convolved with the same spatial blurring kernel in its extended tensor form $\cA\in\mathbb{R}^{m\times m\times n}$. As an illustrative example, we test the 3D MRI image data set \verb"mri" in MATLAB which consists of 12 slices of size $128\times 128$ from an MRI data scan of a human cranium. When $p\ll \min\{m,n\}$, the ground truth $\cX$ can be considered as low-rank. Each blurry image is generated by convolving the ground truth with a Gaussian convolution kernel of size $5\times 5$ with standard deviation 2. The parameters are $\alpha=1$, $\lambda=10^{-2}$ and the maximum iteration number is 1000, and the batch size is 60. Figure~\ref{fig:tendeblur} shows the first four frames of blurry observations and their respective recovered image. To suppress ringing artifacts, projection of all intensities onto the positive values is set as a postprocessing step.
\begin{figure}
\centering\setlength{\tabcolsep}{1pt}
\begin{tabular}{cccc}
\includegraphics[width=.24\textwidth]{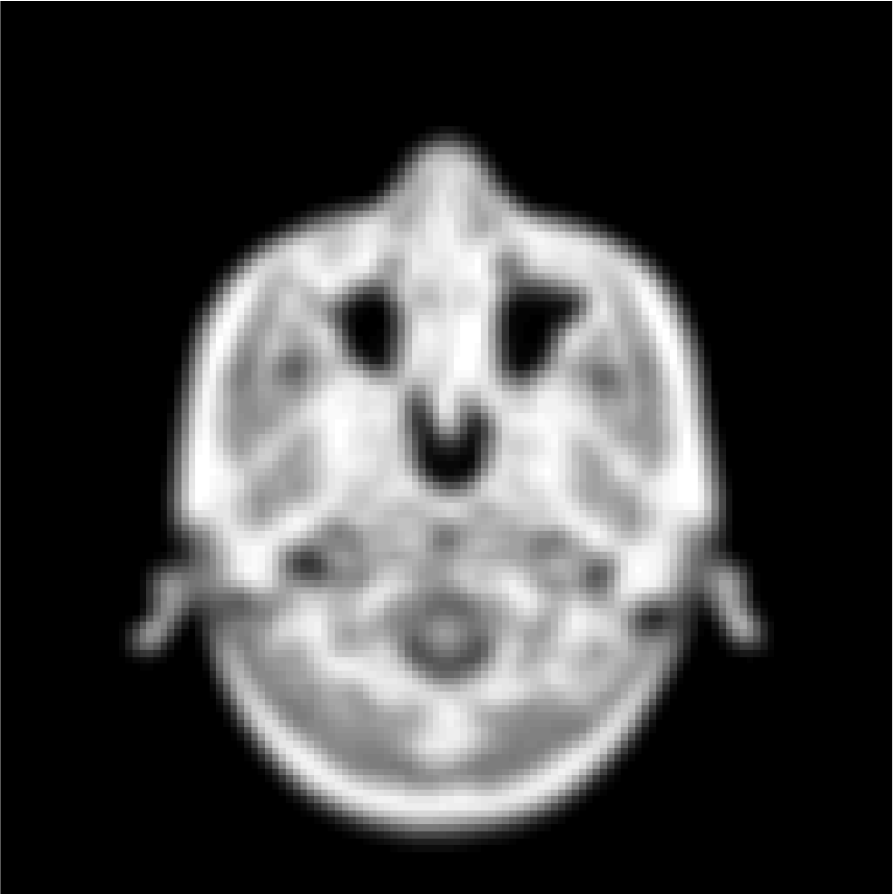}&
\includegraphics[width=.24\textwidth]{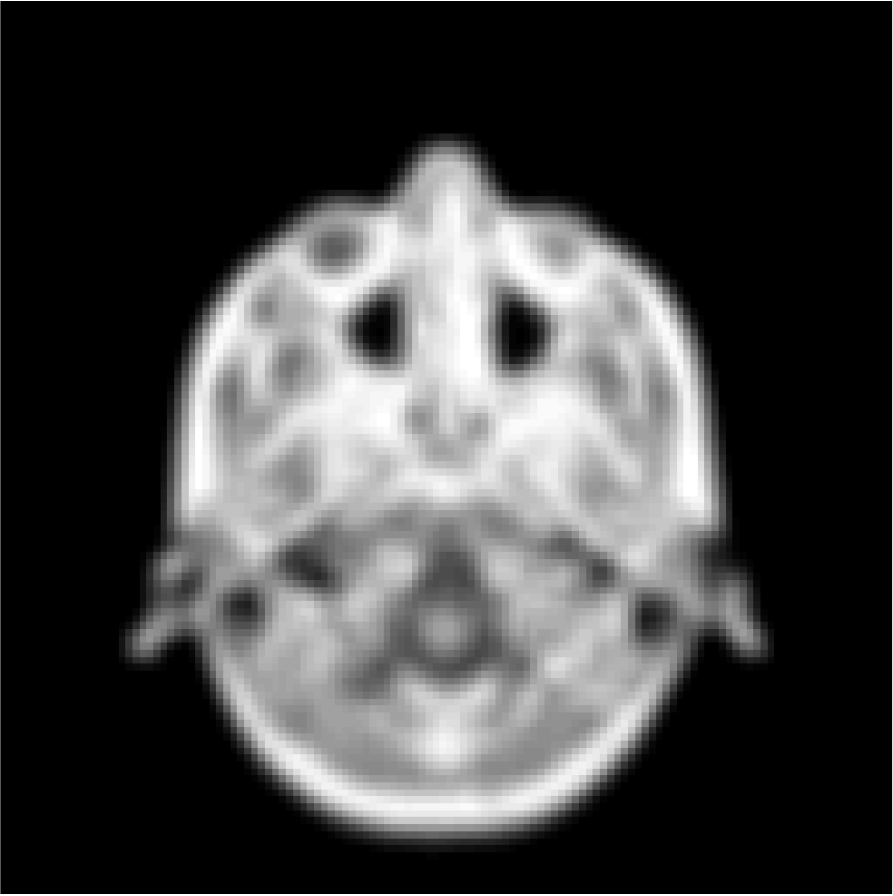}&
\includegraphics[width=.24\textwidth]{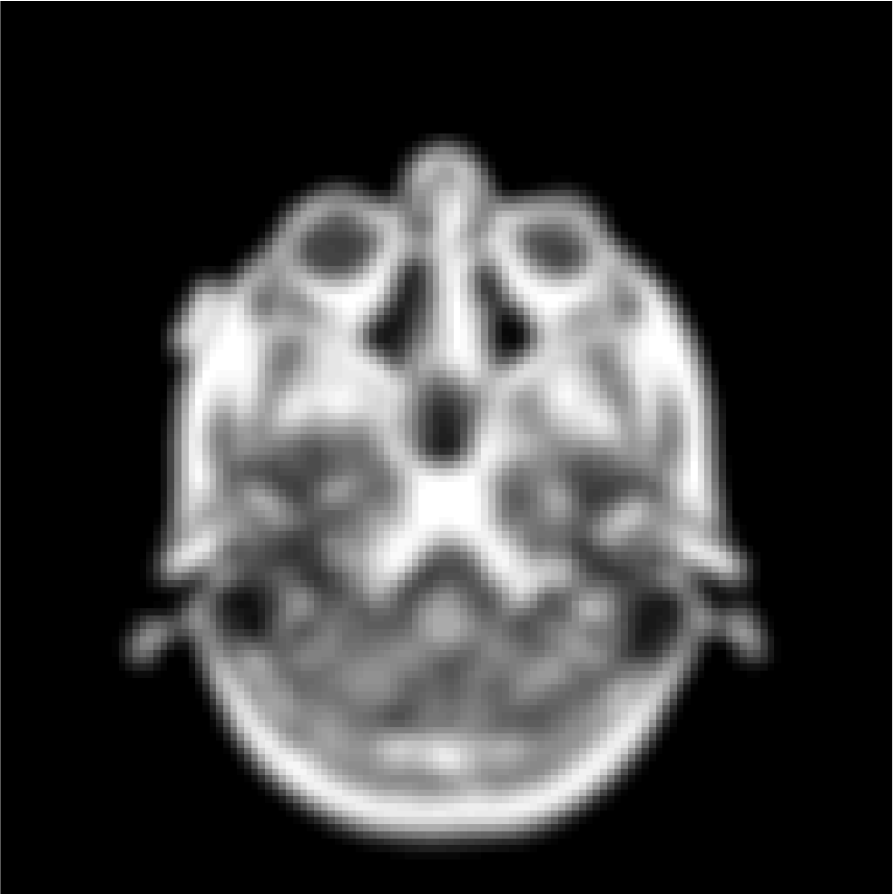}&
\includegraphics[width=.24\textwidth]{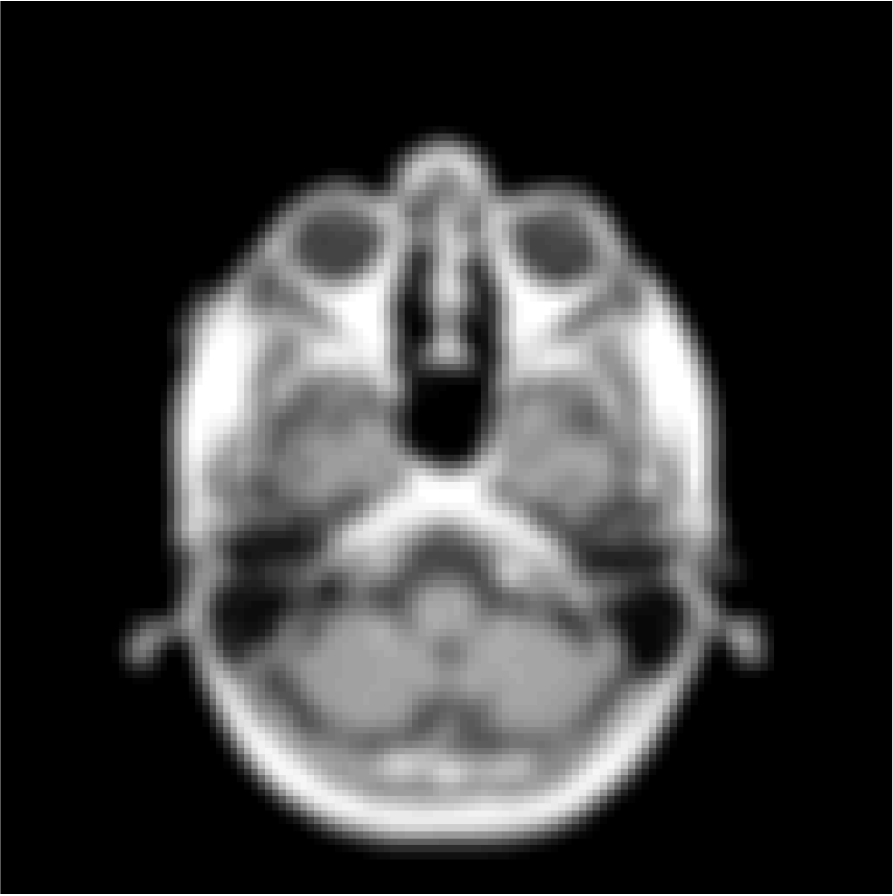}\\
\includegraphics[width=.24\textwidth]{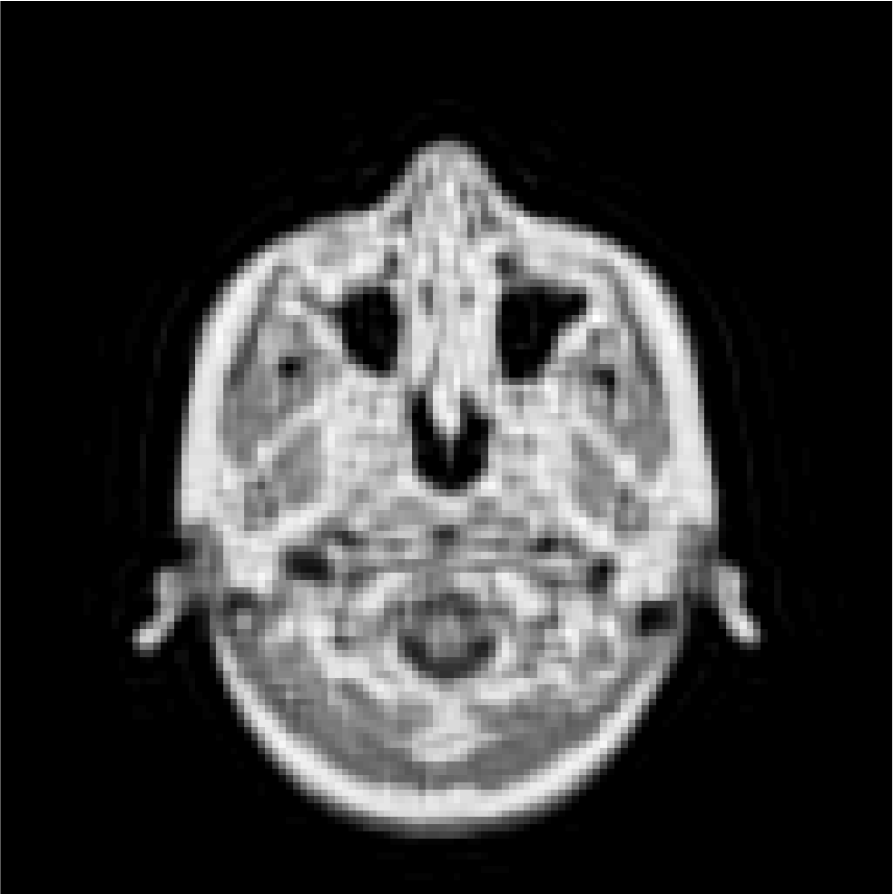}&
\includegraphics[width=.24\textwidth]{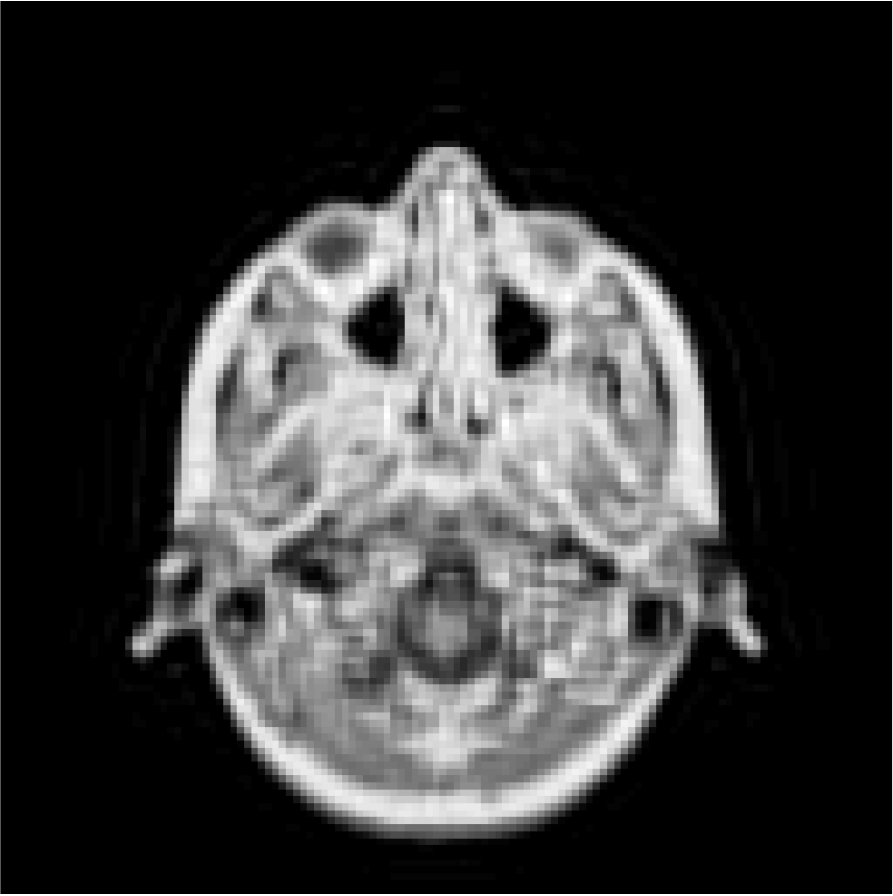}&
\includegraphics[width=.24\textwidth]{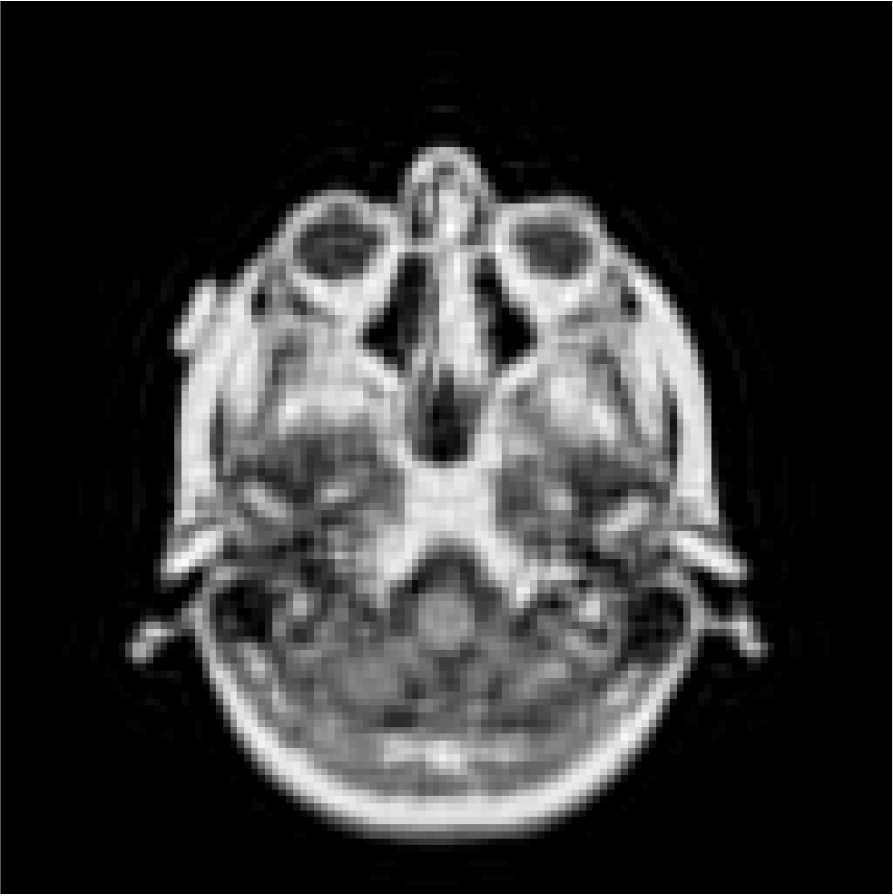}&
\includegraphics[width=.24\textwidth]{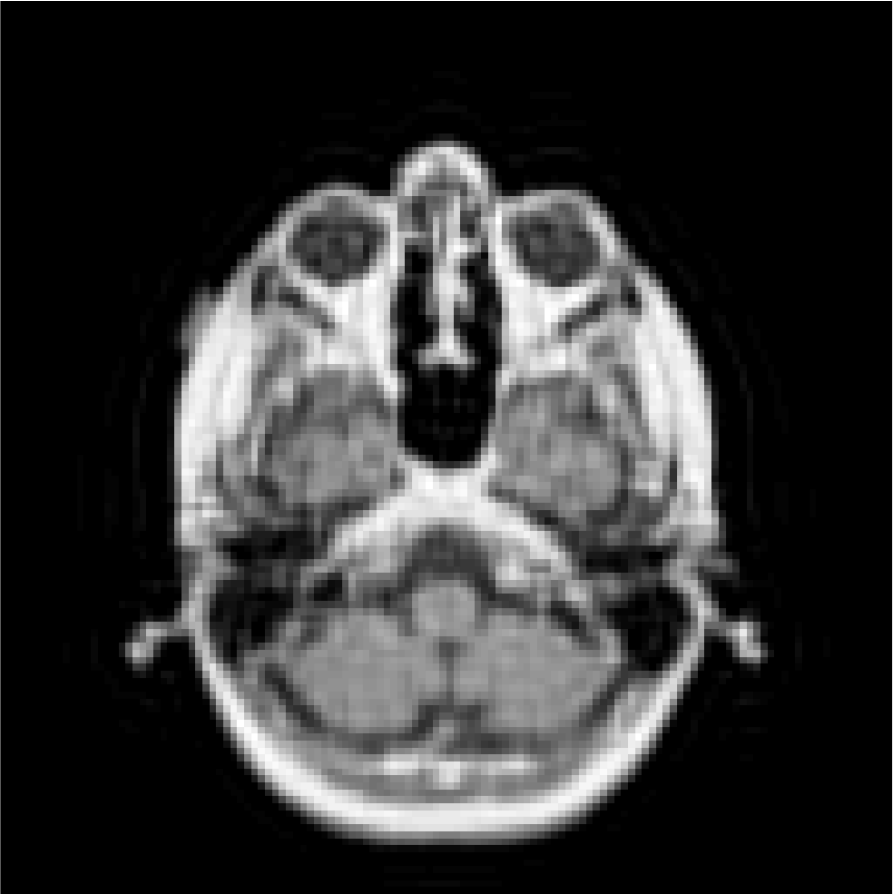}\\
\end{tabular}
\caption{Image Sequence Deblurring. Row 1: blurry images; row 2: recovered images. Overall runtime is 72.90 seconds and the entire relative error between the ground truth and recovered sequences is 13.90\%. }\label{fig:videodeblur}
\end{figure}

\section{Conclusions and Future work}\label{sec:con}
In many tensor recovery problems, the underlying tensor is either sparse or low-rank which can be exploited in the design of efficient algorithms. In this paper, we propose a regularized Kaczmarz algorithm framework for tensor recovery. Precisely, we adopt the t-product for third-order tensors with rapid implementation through the fast Fourier transform, and establish linear convergence rate in expectation for the proposed algorithm with random sequence. In addition, we provide extensive discussions on the matrix and vector recovery together with tensor nuclear norm minimization as special cases. A showcase of numerical experiments demonstrates its considerable potential in various applications, including sparse signal recovery, low-rank tensor recovery, image inpainting and deblurring. In the future, we intend to explore the convergence rate for the deterministic method and discuss theoretical guarantees for its more superior performance over the randomized version. Moreover, it would be extremely intriguing to make a thorough discussion on the acceleration effects by choosing an appropriate batch size or step size. Furthermore, the current framework can be adapted to other types of tensor products or tensors of order higher than three.

\section*{Acknowledgements}
The research of X. Chen is supported by the NSF grant DMS-2050028, and the research of J. Qin is supported by the NSF grant DMS-1941197.

\bibliographystyle{siam}
\bibliography{ref}

\end{document}